\documentclass[a4paper,11pt,twoside]{article}

\usepackage{geometry}
\usepackage[utf8]{inputenc}
\usepackage{lmodern}
\usepackage[T1]{fontenc}

\usepackage{amsmath,amssymb,amsthm,empheq,cases}
\usepackage{hyperref}
\hypersetup{colorlinks,
            citecolor=red, 
            filecolor=black,
            linkcolor=blue,
            urlcolor=black}

\def \dis {\displaystyle}
\def \ecart {\noalign{\medskip}}
\def \RR {\mathbb R}
\def \CC {\mathbb C}
\def \NN {\mathbb N}
\def \L {\mathcal{L}}

\theoremstyle{definition}
\newtheorem{Th}{Theorem}[section]
\newtheorem{Prop}[Th]{Proposition}
\newtheorem{Lem}[Th]{Lemma}
\newtheorem{Cor}[Th]{Corollary}
\newtheorem{Def}[Th]{Definition}
\newtheorem{Rem}[Th]{Remark}

\def \refT #1{Theorem~\ref{#1}}
\def \refL #1{Lemma~\ref{#1}}

\def \refP #1{Proposition~\ref{#1}}
\def \refR #1{Remark~\ref{#1}}

\title{Solvability of a transmission problem in $L^p$-spaces with generalized diffusion equation}

\author{Alexandre Thorel \\
{\scriptsize Université Le Havre Normandie, Normandie Univ, LMAH UR 3821, 76600 Le Havre, France.}\\
{\scriptsize alexandre.thorel@univ-lehavre.fr}}
\date{\empty}

\begin{document}

\maketitle

\begin{abstract}
We study a transmission problem, in population dynamics, between two juxtaposed habitats. In each habitat, we consider a generalized diffusion equation composed by the Laplace operator and a biharmonic term. We, indeed, allow that the coefficients in front of each term could be negative or null. Using semigroups theory and functional calculus, we give some relation between coefficients to obtain the existence and the uniqueness of the classical solution in $L^p$-spaces. 
\end{abstract} 

\section{Introduction}

In this work, using analytic semigroups theory, we study a transmission problem for a coupled system of generalized diffusion equations in $L^p$-spaces, with $p \in (1,+\infty)$. We denote by generalized diffusion equation, an equation of the following form 
$$k \Delta^2 u - l \Delta u = g,$$
with $k,l \in \RR$ and $g$ given. This equation is obtained using the Landau-Ginzburg free energy functional, we refer to \cite{cohen-murray} or \cite{ochoa} for more details. This work is a natural continuation of the works done in \cite{LLMT} and \cite{thorel 2}. Here, we investigate the influence of the Laplace operator and the biharmonic term in the diffusion. In population dynamics, the Laplace operator model the short range diffusion whereas the biharmonic operator represents the long range diffusion. Thus, generalized diffusion is a linear combination of these two operators. 

Usually, in most models $k,l > 0$, but in many works like for instance \cite{cohen-murray}, \cite{gros}, \cite{murray}, \cite{novick-cohen} or \cite{ochoa}, the authors explain that the biharmonic term plays a stabilizing role if $k > 0$ and a destabilizing role when $k < 0$. This is why, in the present paper, we consider $k \in \RR\setminus\{0\}$ and $l \in \RR$.

Many works have treated generalized diffusion equations and transmission problems associated to it. For instance, we refer to \cite{BBT}, \cite{cantin-thorel}, \cite{cohen-murray}, \cite{DMT}, \cite{LMMT}, \cite{LMT}, \cite{Cone 1}, \cite{Cone 2}, \cite{novick-cohen}, \cite{ochoa} and \cite{okubo}, for the study of such an equation in population dynamics and to \cite{denk}, \cite{hassan}, \cite{LLMT} and \cite{thorel 2} for transmission problems associated to it. Note that \cite{LLMT} and \cite{thorel 2}, consider applications in population dynamics wheras \cite{denk} and \cite{hassan}, consider applications in plate theory. 

 
We define $\Omega =\Omega _{-}\cup \Omega _{+}$, the $n$-dimensional area, $n\in \NN\setminus \{0,1\}$, constituted by the two juxtaposed habitats $\Omega_- := (a,\gamma) \times \omega$ and $\Omega_+ := (\gamma,b) \times \omega$ with their interface $\Gamma =\left\{ \gamma \right\} \times \omega$, where $a,\gamma,b \in\mathbb{R}$ with $a<\gamma<b$ and $\omega $ being a bounded domain of $\mathbb{R}^{n-1}$.

We investigate the study of the following transmission problem 
$$(EQ_{pde})\left\{\begin{array}{llll}
k_+ \Delta^2 u_+ - l_+ \Delta u_+ &=& g_+, & \text{in }\Omega_+ \\
k_- \Delta^2 u_- - l_- \Delta u_- &=& g_-, & \text{in }\Omega_-,
\end{array}\right.$$
where $k_\pm \in \RR\setminus\{0\}$, $l_\pm \in \RR$, $u_\pm \in \Omega_\pm$ are population density and $g_\pm \in L^p(\Omega_\pm)$ are given. 

Note that the case $k_\pm,l_\pm > 0$ has been already treated in \cite{LLMT} and the case $l_\pm =0$ with $k_\pm >0$ has been already treated in \cite{thorel 2}. Thus, in the the present article, the most important new results concern the other cases. Indeed, as in the two previous cases, the key point of this article is the inversion of determinant operator but unlike the two previous cases, the writing of this determinant operator by the functional calculus must take into account the sign of the constants $\frac{l_+}{k_+}$ and $\frac{l_-}{k_-}$. Therefore it is necessary to detail each case and for each of these cases, we give conditions between the different constants $k_\pm$ and $l_\pm$ which make it possible to invert this determinant operator.

Here, we denote by $(x,y)$ the spatial variables with $x\in (a,b)$ and $y\in \omega $. The above equations are supplemented by the following boundary and transmission conditions 
\begin{equation*}
(BC_{pde})\left\{\hspace{-0.2cm} \begin{array}{l}
(1)\left\{\begin{array}{rcll}
u_{-}(x,\zeta)&=&0, & x\in (a,\gamma),~ \zeta \in \partial\omega \\
u_{+}(x,\zeta)&=&0, & x\in (\gamma,b),~ \zeta \in \partial\omega \\ 
\dis \Delta u_- (x,\zeta)&=&0, & x\in (a,\gamma),~ \zeta \in \partial\omega \\
\dis \Delta u_+ (x,\zeta)&=&0, & x\in (\gamma,b),~ \zeta \in \partial\omega
\end{array}\right. \\ \ecart
(2) \left\{\begin{array}{rclr}
u_{-}(a,y)&=&\varphi _{1}^{-}(y), & y \in \omega \\ 
u_{+}(b,y)&=&\varphi _{1}^{+}(y), & y \in \omega \\
\dis\frac{\partial u_{-}}{\partial x}(a,y) &=& \varphi_{2}^{-}(y), & y\in \omega \\ \ecart
\dis\frac{\partial u_{+}}{\partial x}(b,y) &=& \varphi_{2}^{+}(y), & y\in \omega,
\end{array}\right.\end{array}\right.
\end{equation*}
where $\varphi_1^\pm$ and $\varphi_2^\pm$ are given in suitable spaces, and
\begin{equation*}
(TC_{pde})\left\{ 
\begin{array}{rcll}
u_{-}&=&u_{+}&\text{on \ }\Gamma \\ 
\dis\frac{\partial u_{-}}{\partial x}&=&\dis\frac{\partial u_{+}}{\partial x} &\text{on \ }\Gamma \medskip \\ 
k_{-}\Delta u_{-}&=&k_{+}\Delta u_{+} & \text{on \ }\Gamma \\ 
\dis\frac{\partial }{\partial x}\left( k_{-}\Delta u_{-} - l_- u_-\right) &=& \dis\frac{\partial }{\partial x}\left( k_{+}\Delta u_{+}-l_{+}u_{+}\right) & \text{on \ }\Gamma.
\end{array}%
\right.
\end{equation*}

In $(BC_{pde})$, the boundary conditions on the two first lines of (1) means that the individuals could not lie on the boundaries $(a,b) \times \partial \omega $, because, for instance, they die or the edge is impassable. The boundary conditions on the two second lines of (1) mean that there is no dispersal in the normal direction. It follows that the dispersal vanishes on $(a,b) \times \partial \omega $. In (2), the population density and the flux are given, for instance on $\left\{ a\right\} \times \omega $ and on $\left\{ b\right\} \times \omega $. This signifies that the habitats are not isolated. Then, in $(TC_{pde})$, the two first transmission conditions mean the continuity of the density and its flux at the interface, while the two second express, in some sense, the continuity of the dispersal and its flux at the interface $\Gamma$.

This article is organized as follows.

In section 2, we give our operational problem. Section 3 is devoted to some recall on BIP operators and real interpolation spaces. In section 4, we give our assumptions and main results. Then, in section 5, we state some preliminary results that will be useful to prove our main result. Finally, section 6, which is composed of three parts, is devoted to the proof of our main result.

\section{Operational formulation}

Since we have
$$\Delta u = \frac{\partial^2 u}{\partial x^2} + \Delta_y u,$$
we set
\begin{equation}\label{def A0}
\left\{\begin{array}{l}
D(A_0) := \{\psi \in W^{2,p}(\omega) : \psi = 0 ~\text{on} ~\partial \omega\} \\ \ecart
\forall \psi \in D(A_0), \quad A_0 \psi = \Delta_y \psi.
\end{array}\right.
\end{equation}
Thus, using operator $A_0$, it follows that each equation of $(EQ_{pde})$ writes
$$u_\pm^{(4)}(x) + (2A_0 - \frac{l_\pm}{k_\pm}\,I) u''_\pm (x) + (A^2_0 - \frac{l_\pm}{k_\pm} \,A_0 ) u_\pm(x) = f_\pm(x),$$
where $u_\pm(x):=u_\pm(x,\cdot)$, $f_\pm(x):=g_\pm(x,\cdot)/k_\pm$ and $f_- \in L^p(a,\gamma;L^p(\omega))$, $f_+ \in L^p(\gamma,b;L^p(\omega))$ with $p \in (1, +\infty)$.

Moreover, since the boundary conditions (1) of $(BC_{pde})$ are included in the domain of $A_0$, problem $(EQ_{pde})-(BC_{pde})-(TC_{pde})$ becomes
\begin{equation*}
\left\{\begin{array}{l}
\begin{array}{lll}
\dis u_+^{(4)}(x) + (2A_0 - \frac{l_+}{k_+}\,I) u''_+ (x) + (A^2_0 - \frac{l_+}{k_+} \,A_0 ) u_+(x) &\hspace*{-0.15cm} = &\hspace*{-0.15cm} f_+(x), \quad \text{for a.e. } x \in (\gamma, b) \\ \ecart
\dis u_-^{(4)}(x) + (2A_0 - \frac{l_-}{k_-}\,I) u''_- (x) + (A^2_0 - \frac{l_-}{k_-} \,A_0 ) u_-(x) &\hspace*{-0.15cm} = &\hspace{-0.15cm} f_-(x), \quad \text{for a.e. } x \in (a,\gamma)
\end{array} \\ \ecart
\begin{array}{llllll}
u_- (a) & = & \varphi_1^-, \quad u_+ (b) & = & \varphi_1^+ \\ \ecart
u'_-(a) & = &\varphi_2^-, \quad u'_+(b) & = & \varphi_2^+
\end{array} \\ \ecart
\begin{array}{lll}
\hspace*{-0.15cm}\begin{array}{lll}
u_-(\gamma) & = & u_+(\gamma) \\ \ecart
u'_-(\gamma) & = & u'_+(\gamma) 
\end{array} \\ \ecart 
k_- u''_-(\gamma)  +  k_- A_0 u_-(\gamma) & = & k_+ u''_+(\gamma)  +  k_+ A_0 u_+(\gamma) \\ \ecart
k_- u_-^{(3)}(\gamma)  +  k_- A_0 u'_-(\gamma) - l_- u'_-(\gamma) & = & k_+ u_+^{(3)}(\gamma)  + k_+ A_0 u'_+(\gamma) - l_+ u'_+(\gamma),
\end{array}\end{array}\right.
\end{equation*}

Then, we will consider a more general case using $(A,D(A))$, instead of $(A_0,D(A_0))$, with $-A$ a BIP operator of angle $\theta_A \in(0,\pi)$ on a UMD space $X$, see below for the definitions of BIP operator and UMD spaces, and $f\in L^p(a,b;X)$. 

More precisely, setting $r_\pm = \frac{l_\pm}{k_\pm}$, we study the following transmission problem (P):
\begin{equation*}
\text{(P)} \left\{\hspace{-0.225cm}\begin{array}{l}
(EQ)\left\{\hspace{-0.1cm}\begin{array}{lll}
\dis u_+^{(4)}(x) + (2A - r_+ \,I) u''_+ (x) + (A^2 - r_+ \,A ) u_+(x) &\hspace*{-0.15cm} = &\hspace*{-0.15cm} f_+(x), \quad x \in (\gamma, b) \\ \ecart
\dis u_-^{(4)}(x) + (2A - r_-\,I) u''_- (x) + (A^2 - r_- \,A ) u_-(x) &\hspace*{-0.15cm} = &\hspace{-0.15cm} f_-(x), \quad x \in (a,\gamma) 
\end{array}\right.\\ \ecart
(BC) \left\{\hspace{-0.1cm}\begin{array}{llllll}
u_- (a) & = & \varphi_1^-, \quad u_+ (b) & = & \varphi_1^+ \\ \ecart
u'_-(a) & = &\varphi_2^-, \quad u'_+(b) & = & \varphi_2^+
\end{array}\right.\\ \ecart
(TC) \left\{\hspace{-0.1cm}\begin{array}{lll}
\hspace*{-0.15cm}\begin{array}{lll}
u_-(\gamma) & = & u_+(\gamma) \\ \ecart
u'_-(\gamma) & = & u'_+(\gamma) 
\end{array} \\ \ecart 
k_- u''_-(\gamma)  +  k_- Au_-(\gamma) & = & k_+ u''_+(\gamma)  +  k_+ Au_+(\gamma) \\ \ecart
k_- u_-^{(3)}(\gamma)  +  k_- Au'_-(\gamma) - l_- u'_-(\gamma) & = & k_+ u_+^{(3)}(\gamma)  + k_+ Au'_+(\gamma) - l_+ u'_+(\gamma).
\end{array}\right.\end{array}\right.
\end{equation*}
The transmission conditions $(TC)$ will be divided into 
$$(TC1)\left\{\begin{array}{lll}
u_-(\gamma) & = & u_+(\gamma) \\ \ecart
u'_-(\gamma) & = & u'_+(\gamma),
\end{array}\right.$$
and
$$(TC2)\left\{\begin{array}{lll}
k_- u''_-(\gamma)  +  k_- Au_-(\gamma) & = & k_+ u''_+(\gamma)  +  k_+ Au_+(\gamma) \\ \ecart
k_- u_-^{(3)}(\gamma)  +  k_- Au'_-(\gamma) - l_- u'_-(\gamma) & = & k_+ u_+^{(3)}(\gamma)  + k_+ Au'_+(\gamma) - l_+ u'_+(\gamma).
\end{array}\right.$$
Note that $(TC2)$ is well defined, see Lemma~\ref{lemtrace} below. 

We will search a classical solution of problem (P), that is a solution $u$ such that 
\begin{equation}\label{u-=u(a,gamma),u+=u(gamma,a)}
\left\{\begin{array}{ll}
u_+:= u_{|(\gamma,b)} \in W^{4,p}(\gamma,b;X)\cap L^p(\gamma,b;D(A^2)), & u''_+ \in L^p(\gamma,b;D(A)), \\ \ecart

u_-:= u_{|(a,\gamma)} \in W^{4,p}(a,\gamma;X)\cap L^p(a,\gamma;D(A^2)), & u''_- \in L^p(a,\gamma;D(A)),
\end{array}\right.
\end{equation}
and which satisfies $(EQ)-(BC)-(TC)$. 

Note that such a solution is not $W^{4,p}(a,b;X)$ but uniquely $W^{4,p}(a,\gamma;X)$ in $\Omega_-$ and $W^{4,p}(\gamma,b;X)$ in $\Omega_+$.

\section{Definitions and prerequisites}

\subsection{The class of Bounded Imaginary Powers of operators}

\begin{Def}
A Banach space $X$ is a UMD space if and only if for all $p\in(1,+\infty)$, the Hilbert transform is bounded from $L^p(\RR,X)$ into itself (see \cite{bourgain} and \cite{burkholder}).
\end{Def}
\begin{Def}\label{Def op Sect}
A closed linear operator $T_1$ is called sectorial of angle $\alpha \in[0,\pi)$ if
$$
\begin{array}{l}
i)\quad \sigma(T_1)\subset \overline{S_{\alpha}},\\\ecart
ii)\quad\forall~\alpha'\in (\alpha,\pi),\quad \sup\left\{\|\lambda(\lambda\,I-T_1)^{-1}\|_{\L(X)} : ~\lambda\in\CC\setminus\overline{S_{\alpha'}}\right\}<+\infty,
\end{array}
$$
where
\begin{equation}\label{defsector}
S_\alpha\;:=\;\left\{\begin{array}{lll}
\left\{ z \in \CC : z \neq 0 ~~\text{and}~~ |\arg(z)| < \alpha \right\} & \text{if} & \alpha \in (0, \pi), \\ \ecart
\,(0,+\infty) & \text{if} & \alpha = 0,
\end{array}\right.
\end{equation} 
see \cite{haase}, p. 19.
\end{Def}
\begin{Rem}
From \cite{komatsu}, p. 342, we know that any injective sectorial operator~$T_1$ admits imaginary powers $T_1^{is}$, $s\in\RR$, but, in general, $T_1^{is}$ is not bounded.
\end{Rem} 
\begin{Def}
Let $\theta \in [0, \pi)$. We denote by BIP$(X,\theta)$, the class of sectorial injective operators $T_2$ such that
\begin{itemize}
\item[] $i) \quad ~~\overline{D(T_2)} = \overline{R(T_2)} = X,$

\item[] $ii) \quad ~\forall~ s \in \RR, \quad T_2^{is} \in \L(X),$

\item[] $iii) \quad \exists~ C \geq 1 ,~ \forall~ s \in \RR, \quad ||T_2^{is}||_{\L(X)} \leq C e^{|s|\theta}$,
\end{itemize}
see~\cite{pruss-sohr}, p. 430.
\end{Def}

\subsection{Interpolation spaces}

Here we recall some properties about real interpolation spaces in particular cases.
\begin{Def}
Let $T_3 : D(T_3) \subset X \longrightarrow X$ be a linear operator such that
\begin{equation}\label{def T3}
(0,+\infty) \subset \rho(T_3)\quad \text{and}\quad \exists~C>0:\forall~t>0, \quad \|t(T_3-tI)^{{\mbox{-}}1}\|_{\L(X)} \leqslant C.
\end{equation}
Let $k \in \NN \setminus \{0\}$, $\theta \in (0,1)$ and $q \in [1,+\infty]$. We will use the real interpolation spaces 
$$(D(T_3^k),X)_{\theta,q} = (X,D(T_3^k))_{1{\mbox{-}}\theta,q},$$ 
defined, for instance, in \cite{lions-peetre}, or in \cite{lunardi}. 

In particular, for $k=1$, we have the following characterization
$$(D(T_3),X)_{\theta,q} := \left\{\psi \in X:t\longmapsto t^{1{\mbox{-}}\theta} \|T_3(T_3-tI)^{{\mbox{-}}1}\psi\|_X \in L_*^q(0,+\infty)\right\},$$
where $L_*^q(0,+\infty)$ is given by
$$L_*^q(0,+\infty;\CC) := \left\{f \in L^q(0,+\infty) : \left(\int_0^{+\infty} |f(t)|^q\frac{dt}{t}\right)^{1/q} < + \infty \right\}, \quad \text{for } q \in [1,+\infty),$$
and for $q = +\infty$, by
$$L_*^\infty(0,+\infty;\CC) := \sup_{t \in (0,+\infty)} |f(t)|,$$
see \cite{da prato-grisvard} p. 325, or \cite{grisvard}, p. 665, Teorema 3, or section 1.14 of \cite{triebel}, where this space is denoted by $(X, D(T_3))_{1{\mbox{-}}\theta,q}$. Note that we can also characterize the space $(D(T_3),X)_{\theta,q}$ taking into account the Osservazione, p. 666, in \cite{grisvard}.

We set also, for any $k \in \NN\setminus\{0\}$
\begin{equation*}
(D(T_3),X)_{k+\theta,q}\;:=\;\left\{\psi\in D(T_3^k) : T_3^k\psi\in (D(T_3),X)_{\theta,q}\right\},
\end{equation*}
and
\begin{equation*}
(X,D(T_3))_{k+\theta,q}\;:=\;\left\{\psi\in D(T_3^k) : T_3^k\psi\in (X,D(T_3))_{\theta,q}\right\},
\end{equation*}
see \cite{lunardi 2}, definition 3.2, p. 64.
\end{Def}

\begin{Rem}
The general situation of the real interpolation space $(X_0,X_1)_{\theta,q}$ with $X_0$, $X_1$ two Banach spaces such that $X_0 \hookrightarrow X_1$, is described in \cite{lions-peetre}.
\end{Rem}
\begin{Rem}\label{Rem Réitération}
Note that for $T_3$ satisfying \eqref{def T3}, $T_3^k$ is closed for any $k \in \NN \setminus \{0\}$ since $\rho(T_3) \neq \emptyset$; consequently, if $k \theta < 1$, we have
$$(D(T_3^k),X)_{\theta,q} = (X,D(T_3^k))_{1-\theta,q} = (X,D(T_3))_{k-k\theta,q} = (D(T_3),X)_{(k-1) + k\theta,q} \subset D(T_3^{k-1}).$$
For more details see \cite{lunardi}, (2.1.13), p. 43 or \cite{grisvard}, p. 676, Teorema 6.
\end{Rem}
We recall the following lemma.
\begin{Lem}[\cite{grisvard}]\label{lemtrace}
Let $T_3$ be a linear operator satisfying \eqref{def T3}. Let $u$ such that
$$u \in W^{n,p}(a_1,b_1;X) \cap L^p(a_1,b_1;D(T_3^k)),$$ 
where $a_1,b_1 \in \RR$ with $a_1<b_1$, $n,k \in \NN\setminus\{0\}$ and $p \in (1,+\infty)$. Then for any $j \in \NN$ satisfying the Poulsen condition $0<\frac{1}{p}+j < n$ and $s \in \{a_1,b_1\}$, we have
$$
u^{(j)}(s) \in (D(T_3^k),X)_{\frac{j}{n}+\frac{1}{np},p}.
$$
\end{Lem}
This result is proved in \cite{grisvard}, p. 678, Teorema 2'.

\section{Assumptions and statement of results}

\subsection{Hypotheses}

In all the sequel, $r_+,r_- \in \RR$, $k_+k_- > 0$ and $A$ denotes a closed linear operator in $X$. We assume the following hypotheses:
\begin{center}
\begin{tabular}{l}
$(H_1)\quad$ $X$ is a UMD space,\\ \ecart
$(H_2)\quad$  $\left[\min(r_+,r_-,0),+\infty\right)\subset\rho(A)$,\\ \ecart
$(H_3)\quad$ $-A \in$ BIP$(X,\theta_A)$ for some $\theta_A \in [0,\pi)$,\\ \ecart
$(H_4)\quad$ $-A \in $ Sect$(0)$.
\end{tabular}
\end{center}

\begin{Rem}\label{Remconsequences}\hfill
\begin{enumerate}
\item Due to $(H_2)$, if at least one parameter $r_+$ or $r_-$ is negative or null, then $0 \in \rho(A)$. 

\item Operator $A_0$, defined by \eqref{def A0}, satisfies all the previous hypotheses with $X=L^q(\omega)$, $q \in (1, +\infty)$ and $r_\pm \in \rho(A_0)$. From \cite{rubio}, Proposition~3, p. 207, $X$ satisfies $(H_1)$ and taking $A_0 + r_\pm I$ in \cite{gilbarg-trudinger}, Theorem~9.15, p.~241 and Lemma~9.17, p. 242, we deduce that $A_0$ satisfies~$(H_2)$. Moreover, $(H_3)$ is satisfied for every $\theta_A \in [0,\pi)$, from \cite{pruss-sohr2}, Theorem~C, p.~166-167. Finally, $(H_4)$ is satisfied thanks to \cite{haase}, section 8.3, p.~232.   

\item In the scalar case, to solve each equation of $(EQ)$, we need to solve the characteristic equations 
$$\chi^4 + (2A - r_\pm)\chi^2 + (A^2 - r_\pm A) = 0,$$ 
thus, in our operational case, we consider the following operators
\begin{equation}\label{defLM}
L_-\;:=\; -\sqrt{-A + r_- \,I}, \quad L_+\;:=\; -\sqrt{-A + r_+ \,I} \quad \text{and} \quad  M\;:=\; -\sqrt{-A}.
\end{equation}
Due to $(H_2)$ and $(H_3)$, $-A$, $-A + r_- \,I$ and $-A+ r_+\, I$ are sectorial operators, so the existence of $L_-$, $L_+$ and $M$ is ensured, see for instance \cite{haase}, e), p. 25 and \cite{arendt-bu-haase}, Theorem~2.3, p.~69.

\item From \cite{haase}, Proposition 3.1.9, p. 65, we have $D(L_-) = D(L_+) = D(M)$. Thus, for $n,m \in \NN$ and $m \leqslant n$
$$D(L^n_\pm) = D(M^n) = D(L^m_\pm M^{n-m}) =  D(M^m L^{n-m}_\pm).$$

\item From \cite{pruss-sohr}, Theorem 3, p. 437 and \cite{arendt-bu-haase}, Theorem 2.3, p. 69, assumptions $(H_2)$ and $(H_3)$ imply that $-A+ r_\pm\,I \in$ BIP$(X,\theta_A)$ and due to \cite{haase}, Proposition $3.2.1,$ e), p. 71, that
$$-L_-,-L_+,-M \in\text{BIP}(X,\theta_A/2).$$
Moreover, from \cite{pruss-sohr}, Theorem 4, p. 441, we get 
$$-(L_-+M),-(L_++M) \in \text{BIP}(X,\theta_A/2+\varepsilon),$$ 
for any $\varepsilon \in (0,\pi/2-\theta_A/2)$.

Since we have $0 < \theta_A/2 < \pi/2$, then due to \cite{pruss-sohr}, Theorem 2, p. 437, we deduce that $L_-$, $L_+$, $M$, $L_- + M$ and $L_+ + M$ generate bounded analytic semigroups $(e^{xL_-})_{x \geqslant 0}$, $(e^{xL_+})_{x \geqslant 0}$, $(e^{xM})_{x \geqslant 0}$, $(e^{x(L_-+M)})_{x \geqslant 0}$ and $(e^{x(L_++M)})_{x \geqslant 0}$.

\item Using the Dore-Venni sums theorem, see \cite{dore-venni}, we deduce from $(H_1)$, $(H_2)$ and $(H_3)$ that $0 \in \rho(M)\cap \rho(L_-)\cap \rho(L_+)\cap \rho(L_++M)\cap \rho(L_-+M)$.

\item From \eqref{defLM}, we deduce that
\begin{equation}\label{L2-M2}
\forall~\psi \in D(M^2), \quad (L_+^2-M^2)\psi = r_+ \,\psi \quad \text{and} \quad (L_-^2-M^2)\psi = r_-\,\psi.
\end{equation}
and also that
\begin{equation}\label{L-M=}
\forall~\psi \in D(M), ~~ (L_+ - M) \psi = r_+ (L_+ + M)^{-1} \psi ~~\text{and} ~~ (L_- - M) \psi = r_- (L_- + M)^{-1} \psi.  
\end{equation}
\end{enumerate}
\end{Rem}

\subsection{Main results}

To solve our operational problem (P), we introduce two problems:
$$(P_+)\left\{\begin{array}{l}
u_+^{(4)}(x) + (2A - r_+\,I) u''_+ (x) + (A^2 - r_+ \,A ) u_+(x) = f_+(x), \quad \text{for a.e. } x \in (\gamma, b) \\ \ecart
\hspace{-0.15cm}\begin{array}{lllllll}
u_+(\gamma) & = & \psi_1, & u_+(b) & = & \varphi_1^+ \\ \ecart
u'_+(\gamma) & = & \psi_2, & u'_+(b) & = & \varphi_2^+.
\end{array}\end{array}\right.$$
and
$$(P_-)\left\{\begin{array}{l}
u_-^{(4)}(x) + (2A - r_-\,I) u''_- (x) + (A^2 - r_- \,A ) u_-(x) = f_-(x), \quad \text{for a.e. } x \in (a, \gamma) \\ \ecart
\hspace{-0.15cm}\begin{array}{lllllll}
u_-(a) & = & \varphi_1^-, & u_-(\gamma) & = & \psi_1 \\ \ecart
u'_-(a) & = & \varphi_2^-, & u'_-(\gamma) & = & \psi_2,
\end{array}\end{array}\right.$$
\begin{Rem}\label{Rem P+ P-} 
Recall that $u$ is a classical solution of (P) if and only if there exist $\psi_1, \psi_2 \in X$ such that
\begin{center}
\begin{tabular}{l}
$~(i)\quad$ ~$u_-$ is a classical solution of $(P_-)$,\\ \ecart
$\,(ii)\quad$ \,$u_+$ is a classical solution of $(P_+)$,\\ \ecart
$(iii)\quad$ $u_-$ and $u_+$ satisfy $(TC2)$.
\end{tabular}
\end{center}
\end{Rem}
Therefore, our aim is to state that there exists a unique couple $(\psi_1,\psi_2)$ which satisfies $(i)$, $(ii)$ and $(iii)$. 

\begin{Th} \label{Th principal}
Let $f_- \in L^p(a,\gamma; X)$ and $f_+ \in L^p(\gamma,b; X)$ with $p \in (1,+\infty)$. Assume that $(H_1)$, $(H_2)$, $(H_3)$, $(H_4)$ hold and $k_+k_- >0$. Thus
\begin{enumerate}
\item when $r_+,r_- \in \RR\setminus\{0\}$,
\begin{itemize}
\item if $r_+ > 0$ and $r_- > 0$,

\item if $r_+ < 0$ and $r_- < 0$, such that 
$$(l_+ - l_-) (k_+ - k_-) \geqslant 0,$$

\item if $r_+ > 0$ and $r_- < 0$, such that 
$$-6 l_-k_+ + l_+k_+ + l_-k_- \geqslant 0,$$

\item if $r_+ < 0$ and $r_- > 0$, such that 
$$-6 l_+k_- + l_+k_+ + l_-k_- \geqslant 0,$$
\end{itemize}

\item when $r_+ \in \RR\setminus\{0\}$ and $r_-=0$ with $\dis\frac{k_-}{k_+} \leqslant 2$,
\begin{itemize}
\item if $r_+ > 0$ such that 
$$r_+ \geqslant \frac{\left(\sqrt{t+1} + \sqrt{t}\right)^2}{t^2} \frac{k_+^2}{4k_-^2},\quad \text{for}~t\in \left(0,\frac{1}{r_+\|A^{-1}\|_{\L(X)}} \right)~\text{fixed},$$

\item if $r_+ < 0$ such that
$$r_+ \leqslant - \frac{27 k_+^2}{64k_-^2},$$ 
\end{itemize}

\item when $r_+ =0$ and $r_-\in \RR\setminus\{0\}$ with $\dis\frac{k_+}{k_-} \leqslant 2$,
\begin{itemize}
\item if $r_- > 0$ such that 
$$r_- \geqslant \frac{\left(\sqrt{t+1} + \sqrt{t}\right)^2}{t^2} \frac{k_-^2}{4k_+^2},\quad \text{for}~t\in \left(0,\frac{1}{r_-\|A^{-1}\|_{\L(X)}} \right)~\text{fixed},$$

\item if $r_- < 0$ such that
$$r_- \leqslant - \frac{27 k_-^2}{64k_+^2},$$ 
\end{itemize}
\end{enumerate}
then, there exists a unique classical solution $u$, of the transmission problem (P) if and only if 
\begin{equation}\label{reg phi +-}
\varphi_1^+, \varphi_1^- \in (D(A),X)_{1 + \frac{1}{2p},p} \quad  \text{and} \quad \varphi_2^+, \varphi_2^- \in (D(A),X)_{1+\frac{1}{2} + \frac{1}{2p},p}.
\end{equation}
\end{Th}

\begin{Rem}
Since the third case is the symmetric of the second one, replacing $k_+$ by $k_-$ and $l_+$ by $l_-$, the proof if exactly the same. Thus, we omit it. 
\end{Rem}

\begin{Rem}
If $A=A_0$, to satisfy the first condition set in the second or the third case of Theorem \ref{Th principal}, since $\|A^{-1}\|_{\L(X)} \geqslant \frac{1}{C_\omega}$, where $C_\omega$ is the Poincar\'e constant, it suffices to take $\omega$ sufficiently small because the more $\omega$ is small, the more $C_\omega$ is large, see for instance \cite{Cantrell-Cosner}, Corollary 2.2, p. 95 and Corollary 2.3, p. 96, or \cite{ACT}, Remark 3, p. 15.
\end{Rem}

As a consequence of the previous Theorem, we state the following corollary.
\begin{Cor}
Let $n \geqslant 2$, $f_+\in L^p(\Omega_+)$ and $f_-\in L^p(\Omega_-)$ with $p\in(1,+\infty)$ and $\dis p > n$. Assume that $\omega$ is a bounded open set of $\RR^{n-1}$ with $C^2$-boundary. Let $k_+,k_-,l_+ > 0$ and $l_- < 0$ with $k_+=k_-$. Then, there exists a unique solution $u$ of $(P_{pde})$, such that we have $u_-\in W^{4,p}(\Omega_-)$ and $u_+\in W^{4,p}(\Omega_+)$, if and only if 
$$\left\{\begin{array}{l}
\varphi_1^\pm, \varphi_2^\pm \in  W^{2,p}(\omega)\cap W^{1,p}_0(\omega) \\ \ecart
\Delta \varphi_1^\pm, \in W^{2-\frac{1}{p},p}(\omega) \cap W_0^{1,p}(\omega) \\ \ecart
\Delta\varphi_2^\pm \in W^{1-\frac{1}{p},p}(\omega) \cap W_0^{1,p}(\omega).
\end{array}\right.$$
\end{Cor}
The proof is quite similar to the one stated in \cite{LLMT}, Corollary 1, p. 2941, or in \cite{LMMT}, Corollary 2.7, p. 357. Thus we omit it. 

\section{Preliminary results}

In all the sequel, we set 
$$c = \gamma-a > 0\quad \text{and} \quad d = b-\gamma > 0.$$ 
From \refR{Rem P+ P-}, to solve problem (P) we must first study problems $(P_+)$ and $(P_-)$. To this end, we need the following invertibility result obtained in \cite{LMMT} and \cite{thorel}. 
\begin{Lem}[\cite{LMMT} and \cite{thorel}]\label{lem U+- V+- inv}
Assume that $(H_1)$, $(H_2)$, $(H_3)$ and $(H_4)$ hold. Then operators $U_\pm$, $V_\pm \in \L(X)$ defined by
\begin{equation} \label{U V +-}
\left\{ \begin{array}{rcl}
U_+ &:=& \left\{\begin{array}{ll}
\dis I - e^{d(L_+ + M)} - \frac{1}{r_+} (L_+ + M)^2 \left(e^{dM} - e^{dL_+}\right), & \text{if } r_+ \in \RR \setminus\{0\} \\
\dis I - e^{2d M} + 2d Me^{d M}, & \text{if } r_+ = 0
\end{array}\right.  \\ \ecart

V_+ &:=& \left\{\begin{array}{ll}
\dis I - e^{d(L_+ + M)} + \frac{1}{r_+} (L_+ + M)^2 \left(e^{dM} - e^{dL_+}\right), & \text{if } r_- \in \RR \setminus\{0\} \\ 
\dis I - e^{2d M} - 2d Me^{d M}, & \text{if } r_+ = 0
\end{array}\right.  \\ \ecart

U_- &:=& \left\{\begin{array}{ll}
\dis I - e^{c(L_- + M)} - \frac{1}{r_-} (L_- + M)^2 \left(e^{cM} - e^{cL_-}\right), & \text{if } r_- \in \RR \setminus\{0\} \\

\dis I - e^{2c M} + 2c Me^{c M}, & \text{if } r_- = 0
\end{array}\right.  \\ \ecart

V_- &:=& \left\{\begin{array}{ll}
\dis I - e^{c(L_- + M)} + \frac{1}{r_-} (L_- + M)^2 \left(e^{cM} - e^{cL_-}\right), & \text{if } r_- \in \RR \setminus\{0\} \\ \ecart
\dis I - e^{2c M} - 2c Me^{c M}, & \text{if } r_- = 0,
\end{array}\right. 
\end{array}\right.
\end{equation}
are invertible with bounded inverses.
\end{Lem}
From \refR{Remconsequences}, statement $4$, $U_\pm$ and $V_\pm$ are well defined. For a detailed proof, see \cite{LMMT}, Proposition 5.4 with $k=r_\pm$ and \cite{thorel}, Proposition 4.5, p. 645. 

\subsection{Transmission system}

\subsubsection{First case}

Assume that $r_\pm \in \RR\setminus \{0\}$. We set
\begin{equation}\label{Pi+}
\left \{ \begin{array}{rcl}
P_1^+ & = & \dis k_+ (L_+ + M) \left(U^{-1}_+(I + e^{dM})(I - e^{dL_+}) +  V^{-1}_+(I - e^{dM})(I+e^{dL_+}) \right) \\ \ecart

P_2^+ & = & \dis k_+ (L_+ + M) \left(U^{-1}_+ (I - e^{dM})(I - e^{dL_+}) + V^{-1}_+ (I+e^{dM})(I + e^{dL_+}) \right) \\ \ecart

P_3^+ & = & \dis  k_+ (L_+ + M) L_+ \left(U^{-1}_+(I + e^{dM})(I + e^{dL_+}) + V^{-1}_+(I - e^{dM})(I - e^{dL_+}) \right),
\end{array}\right.
\end{equation}
and similarly
\begin{equation}\label{Pi-}
\left \{ \begin{array}{rcl}
P_1^- & = & \dis k_- (L_- + M) \left(U^{-1}_-(I + e^{cM})(I - e^{cL_-}) + V^{-1}_-(I - e^{cM}) (I+e^{cL_-}) \right) \\ \ecart

P_2^- & = & \dis k_- (L_- + M)\left(U^{-1}_- (I - e^{cM})(I - e^{cL_-}) + V^{-1}_- (I+e^{cM})(I + e^{cL_-}) \right) \\ \ecart

P_3^- & = & \dis k_- (L_- + M) L_- \left( U^{-1}_-(I + e^{cM})(I + e^{cL_-}) + V^{-1}_-(I - e^{cM}) (I - e^{cL_-}) \right).
\end{array}\right.
\end{equation}
Moreover, we note
\begin{equation}\label{S2}
\begin{array}{lll}
S_1 & = & \dis k_+ (L_+ + M)\left( U^{-1}_+(I - e^{dL_+}) \tilde{\varphi_2}^+ + V^{-1}_+(I + e^{dL_+}) \tilde{\varphi_4}^+ \right) \\ \ecart
&& \dis - k_- (L_- + M) \left(U^{-1}_-(I - e^{cL_-}) \tilde{\varphi_2}^- + V^{-1}_-(I + e^{cL_-}) \tilde{\varphi_4}^-\right),
\end{array}
\end{equation}
and
\begin{equation}\label{S1}
\begin{array}{lll}
S_2 & = & \dis -k_+ (L_+ + M) \left( U^{-1}_+(I + e^{dM}) \tilde{\varphi_1}^+ + V^{-1}_+ (I - e^{dM}) \tilde{\varphi_3}^+\right) \\ \ecart
&& \dis - k_- (L_- + M)  \left(U^{-1}_-(I + e^{cM}) \tilde{\varphi_1}^- + V^{-1}_-(I - e^{cM}) \tilde{\varphi_3}^- \right) - 2 M^{-1} R_1,
\end{array}
\end{equation}
with  
\begin{equation}\label{R1}
R_1 = - k_+ F'''_+(\gamma) + k_+ M^2 F'_+(\gamma) + l_+ F'_+(\gamma) + k_- F'''_-(\gamma) - k_- M^2 F'_-(\gamma) - l_- F'_-(\gamma),
\end{equation}
where $F_+$ is the unique classical solution of problem
\begin{equation}\label{pb F+} 
\left\{\begin{array}{l}
u_+^{(4)}(x) + (2A - r_+\,I) u''_+ (x) + (A^2 - r_+ \,A ) u_+(x) = f_+(x), \quad \text{for a.e. } x \in (\gamma, b) \\ \ecart
u_+(\gamma) = u_+(b) = u''_+(\gamma) = u''_+(b) = 0,
\end{array}\right.
\end{equation}
and $F_-$ is the unique classical solution of problem
\begin{equation} \label{pb F-} 
\left\{\begin{array}{l}
u_-^{(4)}(x) + (2A - r_-\,I) u''_- (x) + (A^2 - r_- \,A ) u_-(x) = f_-(x), \quad \text{for a.e. } x \in (a,\gamma) \\ \ecart
u_-(a) = u_-(\gamma) = u''_-(a) = u''_+(\gamma) = 0.
\end{array}\right.
\end{equation}
For an explicit representation formula of the solution of both previous problems, we refer to \cite{LMMT}, Theorem 2.2, p.~355-356.
\begin{Rem}\label{Rem F+-}
Since $F_\pm$ is a classical solution of \eqref{pb F+}, respectively \eqref{pb F-}, from \refL{lemtrace}, it follows that, for $j = 0,1,2,3$ and $s = a$, $\gamma$ or $b$
$$F_\pm^{(j)}(s) \in (D(M),X)_{3-j + \frac{1}{p},p}.$$
\end{Rem}
Now, with our notations, we recall a useful result of \cite{LLMT}, Theorem~4.6, p. 2945. This result is proved for $r_\pm > 0$ but clearly remains true for $r_\pm < 0$.
\begin{Th}[\cite{LLMT}]\label{Th syst trans}
Let $f_- \in L^p(a,\gamma; X)$ and $f_+ \in L^p(\gamma,b; X)$, with $p \in (1,+\infty)$. Assume that $(H_1)$, $(H_2)$, $(H_3)$ and $(H_4)$ hold. Then, the transmission problem $(P)$ has a unique classical solution if and only if the data $\varphi_1^+, \varphi_1^-, \varphi_2^+, \varphi_2^-$ satisfy \eqref{reg phi +-} and system 
\begin{equation}\label{syst trans gén}
\left\{\begin{array}{cclll}
\left(P_1^- - P_1^+ \right) M \psi_1 &+& \left(P_2^+ + P_2^-\right) \psi_2 &=& S_1 \\ \ecart
\left(P_3^+ + P_3^- \right) \psi_1 &+& \left(P_1^- - P_1^+ \right) \psi_2 &=& S_2,
\end{array}\right.
\end{equation} 
has a unique solution $(\psi_1,\psi_2)$ such that 
\begin{equation}\label{reg (psi1,psi2)}
(\psi_1,\psi_2) \in (D(A),X)_{1+\frac{1}{2p},p} \times (D(A),X)_{1+\frac{1}{2}+\frac{1}{2p},p}.
\end{equation}
\end{Th}

\subsubsection{Second case}

Now, assume that $r_- = 0$; then $l_-=0$ and the transmission conditions $(TC2)$ becomes
\begin{equation}\label{TC2'}
(TC2')\left\{\begin{array}{lll}
k_- u''_-(\gamma)  +  k_- Au_-(\gamma) & = & k_+ u''_+(\gamma)  +  k_+ Au_+(\gamma) \\ \ecart
k_- u_-^{(3)}(\gamma)  +  k_- Au'_-(\gamma) & = & k_+ u_+^{(3)}(\gamma)  + k_+ Au'_+(\gamma) - l_+ u'_+(\gamma).
\end{array}\right.
\end{equation}
Our aim here is to establish a similar result to the previous one. To this end, we set
\begin{equation}\label{Qi-}
\left\{\begin{array}{lcl}
Q_1^- & = & k_- \left(U_-^{-1} + V_-^{-1} \right)  \left(I - e^{2cM}\right) \\ \ecart

Q_2^- & = & k_- \left(U_-^{-1}\left(I - e^{cM}\right)^2 + V_-^{-1} \left(I + e^{cM}\right)^2 \right) \\ \ecart

Q_3^- & = & k_- \left(U_-^{-1} \left(I + e^{cM}\right)^2 + V_-^{-1} \left(I - e^{cM}\right)^2\right),
\end{array}\right.
\end{equation}
Moreover, we note  
\begin{equation}\label{S3}
\begin{array}{lll}
S_3 & = & 2 k_- M \left(U_-^{-1} \left(I - e^{c M} \right) \tilde{\varphi_2}^- + V_-^{-1} \left(I + e^{c M} \right) \tilde{\varphi_4}^- \right) \\ \ecart
&& - k_+ (L_+ + M) \left(U_+^{-1}\left(I - e^{dL_+}\right) \tilde{\varphi_2}^+ + V_+^{-1} \left(I + e^{dL_+}\right) \tilde{\varphi_4}^+ \right),
\end{array}
\end{equation}
and
\begin{equation}\label{S4}
\begin{array}{lll}
S_4 & = & - 2 k_- M \left(U_-^{-1}\left(I + e^{c M}\right)\tilde{\varphi}_2^- + V_-^{-1} \left(I - e^{c M}\right)\tilde{\varphi}_4^-\right) \\ \ecart
&& - k_+ (L_+ + M) \left(U_+^{-1} \left(I + e^{dM}\right) \tilde{\varphi}_1^+ + V_+^{-1}\left(I - e^{dM}\right)\tilde{\varphi}_3^+\right) + 2 M^{-1} R_2,
\end{array}
\end{equation}
with
\begin{equation}\label{R2}
R_2 = - k_- \tilde{F}'''_-(\gamma) + k_- M^2 \tilde{F}'_-(\gamma) + k_+ F'''_+(\gamma) -  k_+ M^2 F'_+(\gamma) - l_+ F'_+(\gamma),
\end{equation}
where $\tilde{F}_-$ is the classical solution of problem
\begin{equation}\label{pb F- M} 
\left\{\begin{array}{l}
u_-^{(4)}(x) + 2A u''_- (x) + A^2 u_-(x) = f_-(x), \quad \text{a.e. } x \in (a, \gamma) \\ \ecart
u_-(a) = u_-(\gamma) = u''_-(a) = u''_-(\gamma) = 0.
\end{array}\right.
\end{equation}
\begin{Rem}\label{Rem tilde F-}
Since $\tilde{F}_-$ is a classical solution of \eqref{pb F-}, as in \refR{Rem F+-}, from \refL{lemtrace}, it follows that, for $j = 0,1,2,3$ and $s = a$, $\gamma$ or $b$
$$\tilde{F}_-^{(j)}(s) \in (D(M),X)_{3-j + \frac{1}{p},p}.$$
\end{Rem}

\begin{Th}\label{Th syst trans M}
Let $f_- \in L^p(a,\gamma;X)$ and $f_+ \in L^p(\gamma,b;X)$, with $p \in (1,+\infty)$. Assume that $(H_1)$, $(H_2)$, $(H_3)$ and $(H_4)$ hold. Then, problem (P) has a unique classical solution if and only if the data $\varphi_1^+$, $\varphi_1^-$, $\varphi^+_2$, $\varphi^-_2$ satisfy \eqref{reg phi +-} and system  
\begin{equation}\label{syst trans M}
\left\{\begin{array}{cclll}
\left(P_1^+ - 2M Q_1^- \right) M \psi_1 &-& \left(P_2^+ + 2M Q_2^-\right) \psi_2 &=& S_3 \\ \ecart
\left(P_3^+ + 2M Q_3^- \right) \psi_1 &+& \left(2M Q_1^- - P_1^+ \right) \psi_2 &=& S_4,
\end{array}\right.
\end{equation} 
has a unique solution $(\psi_1,\psi_2)$ satisfying \eqref{reg (psi1,psi2)}.
\end{Th}
\begin{proof}
We follow the same steps than the proof of Theorem 4.6, p. 2945 in \cite{LLMT}, we only point out the key points. From \cite{LMMT}, Theorem 2.5, statement 2, there exists a unique classical solution $u_+$ of $(P_+)$ if and only if 
\begin{equation}\label{reg phi/psi i P+}
\varphi_1^+,\psi_1 \in (D(A),X)_{1+\frac{1}{2p},p} \quad  \text{and} \quad \varphi_2^+,\psi_2 \in (D(A),X)_{1+\frac{1}{2}+\frac{1}{2p},p}.
\end{equation}
Recall that, from \refR{Rem Réitération}, we have 
\begin{equation}\label{égalités espaces interpolations}
(D(A),X)_{1+\frac{1}{2p},p} = (D(M),X)_{3+\frac{1}{p},p} \quad \text{and} \quad (D(A),X)_{1+\frac{1}{2}+\frac{1}{2p},p} = (D(M),X)_{2+\frac{1}{p},p}.
\end{equation}
This solution is explicitly given in \cite{LLMT}, Proposition 2, p. 2943-2944, from which we deduce that
$$k_+\left(u''_+(\gamma) - M^2 u_+(\gamma)\right) = l_+ \left(I - e^{dL_+}\right) \alpha_2^+ + l_+ \left(I + e^{dL_+}\right) \alpha_4^+,$$
and
$$\begin{array}{lll}
k_+\left(u^{(3)}_+(\gamma) - M^2 u'_+(\gamma)\right) - l_+ u'_+(\gamma) & = & \dis - l_+ M\left(I + e^{dM}\right) \alpha_1^+ - l_+ M\left(I - e^{dM}\right) \alpha_3^+ \\ \ecart
&& \dis + k_+ F'''_+(\gamma) -  k_+ M^2 F'_+(\gamma) - l_+ F'_+(\gamma),
\end{array}$$
where
\begin{equation}\label{ali+} 
\left\{\begin{array}{rcr}
\alpha_1^+ & = & \dis \frac{1}{2r_+}(L_+ + M)U_+^{-1}\left[L_+(I+e^{dL_+})\psi_1 - (I-e^{dL_+})\psi_2 + \tilde{\varphi}_1^+\right] \\ \ecart

\alpha_2^+ & =& \dis -\frac{1}{2r_+}(L_+ + M)U_+^{-1}\left[M(I+e^{dM})\psi_1 - (I-e^{dM})\psi_2 + \tilde{\varphi}_2^+\right] \\ \ecart

\alpha_3^+& = & \dis \frac{1}{2r_+}(L_+ + M)V_+^{-1}\left[L_+(I-e^{dL_+})\psi_1 - (I+e^{dL_+}) \psi_2 + \tilde{\varphi}_3^+\right] \\ \ecart

\alpha_4^+ & = & \dis -\frac{1}{2r_+}(L_+ + M)V_+^{-1}\left[M(I-e^{dM})\psi_1 - (I+e^{dM}) \psi_2 + \tilde{\varphi}_4^+\right],
\end{array}\right.
\end{equation}
with
\begin{equation} 
\label{phi i tilde +} \left\{\begin{array}{rcr}
\tilde{\varphi_1}^+ & = & - L_+ \left(I+e^{dL_+}\right)\varphi_1^+ + \left(I-e^{dL_+}\right)\left(F'_+(b) + F'_+(\gamma) - \varphi_2^+\right) \\ \ecart

\tilde{\varphi_2}^+ & = & -M \left(I+e^{dM}\right)\varphi_1^+ + \left(I-e^{dM}\right)\left(F'_+(b) + F'_+(\gamma) - \varphi_2^+\right) \\ \ecart

\tilde{\varphi_3}^+ & = & L_+ \left(I-e^{dL_+}\right)\varphi_1^+ - \left(I+e^{dL_+}\right) \left(F'_+(b) - F'_+(\gamma) - \varphi_2^+ \right) \\ \ecart

\tilde{\varphi_4}^+ & = & M \left(I-e^{dM}\right)\varphi_1^+ - \left(I+e^{dM}\right) \left(F'_+(b) - F'_+(\gamma) - \varphi_2^+ \right),
\end{array}\right.
\end{equation}
and $F_+$ is the unique classical solution of problem \eqref{pb F+}.

In the same way, from \cite{thorel}, Theorem 2.8, statement 2, p. 637, there exists a unique classical solution $u_-$ of problem $(P_-)$ if and only if 
\begin{equation}\label{reg phi/psi i P-}
\varphi_1^-,\psi_1 \in (D(A),X)_{1+\frac{1}{2p},p} \quad  \text{and} \quad \varphi_2^-,\psi_2 \in (D(A),X)_{1+\frac{1}{2}+\frac{1}{2p},p}.
\end{equation}
Note that from \eqref{égalités espaces interpolations}, we have
\begin{equation*}
\varphi_1^-,\psi_1 \in (D(M),X)_{3+\frac{1}{p},p} \quad \text{and} \quad \varphi_2^-,\psi_2 \in (D(M),X)_{2+\frac{1}{p},p}.
\end{equation*}
Moreover, this solution, given in \cite{thorel}, Proposition 4.1, p. 640, is explicitly written in \cite{thorel 2}, Proposition 4.2, from which it follows that
$$k_-\left(u''_-(\gamma) - M^2 u_-(\gamma)\right) = -k_-\left(2M\left(I - e^{c M} \right) \alpha_2^- - 2M\left(I + e^{c M} \right) \alpha_4^-\right),$$
and
$$\begin{array}{lllll}
\dis k_-\left(u^{(3)}_-(\gamma) - M^2 u'_-(\gamma)\right) &=& \dis k_- \left(2M^2\left(I + e^{c M}\right) \alpha_2^- - 2 M^2 \left(I - e^{c M}\right) \alpha_4^-\right) \\ \ecart
&& \dis + k_- F^{(3)}_-(\gamma) - k_- M^2 F'_-(\gamma),
\end{array}$$
where
\begin{equation}\label{ali- M}
\left \{ \begin{array}{rcr}
\alpha_1^- & := & \dis -\frac{1}{2} U_-^{-1} \left[\left(I + \left( I + c M \right) e^{cM} \right)\psi_1 - ce^{cM} \psi_2 + \tilde{\varphi}_1^- \right] \\ \ecart

\alpha_2^- & := & \dis \frac{1}{2} U_-^{-1} \left[ \left(I + e^{cM}\right) M \psi_1 + \left(I - e^{cM}\right) \psi_2 + \tilde{\varphi}_2^- \right] \\ \ecart

\alpha_3^- & := & \dis \frac{1}{2} V_-^{-1} \left[\left(I - \left( I + c M \right) e^{cM} \right) \psi_1 + c e^{cM} \psi_2 + \tilde{\varphi}_3^- \right] \\ \ecart

\alpha_4^- & := & \dis - \frac{1}{2} V_-^{-1} \left[ \left(I - e^{cM}\right) M \psi_1 + \left(I + e^{cM}\right) \psi_2 + \tilde{\varphi}_4^-\right],
\end{array}\right.
\end{equation}
with
\begin{equation} \label{phi i tilde - M}
\left\{\begin{array}{rcr}
\tilde{\varphi_1}^- & := & \dis - \left(I + e^{cM}\right) \varphi_1^- - c e^{cM} \left( M \varphi_1^- + \varphi_2^- - \tilde{F}_-'(a) - \tilde{F}_-'(\gamma)\right) \\ \\

\tilde{\varphi_2}^- & := & \dis - M \left( I + e^{cM}\right) \varphi_1^- + \left( I - e^{cM}\right) \left( \varphi_2^- - \tilde{F}_-'(a) - \tilde{F}_-'(\gamma)\right) \\ \\

\tilde{\varphi_3}^- & := & \dis \left( I - e^{cM}\right) \varphi_1^- - c e^{cM} \left(M \varphi_1^- + \varphi_2^- - \tilde{F}_-'(a) + \tilde{F}_-'(\gamma)\right) \\ \\
 
\tilde{\varphi_4}^- & := & \dis M \left( I - e^{cM}\right) \varphi_1^- - \left( I + e^{cM}\right)\left(\varphi_2^- - \tilde{F}_-'(a) + \tilde{F}_-'(\gamma)\right).
\end{array}\right.
\end{equation}
Note that due to \eqref{reg phi/psi i P+}, \eqref{égalités espaces interpolations}, \eqref{ali+} and \eqref{phi i tilde +}, respectively to \eqref{égalités espaces interpolations}, \eqref{reg phi/psi i P-}, \eqref{ali- M} and \eqref{phi i tilde - M}, we deduce that
$$\alpha^\pm_i \in D(M), \quad \text{for}~ i=1,2,3,4 \quad \text{and} \quad \alpha_2^-,\alpha_4^- \in D(M^2).$$
Thus, system $(TC2')$, given by \eqref{TC2'}, writes
\begin{equation*} 
\left\{\hspace{-0.215cm}\begin{array}{rcl}
-2 k_- M \left(\hspace{-0.05cm}\left(I - e^{c M} \right) \alpha_2^- - \left(I + e^{c M} \right) \alpha_4^-\right) &\hspace{-0.25cm} = &\hspace{-0.25cm} l_+ \left(\hspace{-0.05cm}\left(I - e^{dL_+}\right) \alpha_2^+ + \left(I + e^{dL_+}\right) \alpha_4^+ \right) \\ \\

2 k_- M^2 \left(\hspace{-0.05cm}\left(I + e^{c M}\right) \alpha_2^- - \left(I - e^{c M}\right) \alpha_4^-\right) &\hspace{-0.25cm} = &\hspace{-0.25cm} - l_+ M \left(\hspace{-0.05cm}\left(I + e^{dM}\right) \alpha_1^+ + \left(I - e^{dM}\right) \alpha_3^+\right)\hspace{-0.05cm} +\hspace{-0.05cm} R_2, 
\end{array}\right.
\end{equation*}
where $R_2$ is given by \eqref{R2}. Thus, it follows that the previous system gives
\begin{equation*} 
\left\{\begin{array}{rcl}
- 2 k_- U_-^{-1}  M \left(I - e^{c M} \right)  \left[ \left(I + e^{cM}\right) M \psi_1 + \left(I - e^{cM}\right) \psi_2 + \tilde{\varphi}_2^- \right] \\ \ecart
- 2 k_- V_-^{-1} M \left(I + e^{c M} \right) \left[ \left(I - e^{cM}\right) M \psi_1 + \left(I + e^{cM}\right) \psi_2 + \tilde{\varphi}_4^-\right] \\ \ecart
+ k_+ (L_+ + M)U_+^{-1}\left(I - e^{dL_+}\right)\left[M(I+e^{dM})\psi_1 - (I-e^{dM})\psi_2 + \tilde{\varphi}_2^+\right] \\ \ecart
+ k_+ (L_+ + M)V_+^{-1} \left(I + e^{dL_+}\right) \left[M(I-e^{dM})\psi_1 - (I+e^{dM}) \psi_2 + \tilde{\varphi}_4^+\right] & = & 0 \\ \\

2 k_- M U_-^{-1}\left(I + e^{c M}\right) \left[ \left(I + e^{cM}\right) M \psi_1 + \left(I - e^{cM}\right) \psi_2 + \tilde{\varphi}_2^- \right] \\ \ecart 
+ 2 k_- M V_-^{-1} \left(I - e^{c M}\right) \left[ \left(I - e^{cM}\right) M \psi_1 + \left(I + e^{cM}\right) \psi_2 + \tilde{\varphi}_4^-\right] \\ \ecart
+ k_+ (L_+ + M)U_+^{-1} \left(I + e^{dM}\right) \left[L_+(I+e^{dL_+})\psi_1 - (I-e^{dL_+})\psi_2 + \tilde{\varphi}_1^+\right] \\ \ecart
+ k_+ (L_+ + M)V_+^{-1}\left(I - e^{dM}\right) \left[L_+(I-e^{dL_+})\psi_1 - (I+ e^{dL_+}) \psi_2 + \tilde{\varphi}_3^+\right] & = & 2 M^{-1} R_2, 
\end{array}\right.
\end{equation*} 
Finally, using \eqref{Pi+}, \eqref{Qi-}, \eqref{S3}, \eqref{S4} and \eqref{R2}, we obtain that the previous system writes as system \eqref{syst trans M}.

Conversely, if we assume that \eqref{reg phi +-} holds and system \eqref{syst trans M} has a unique solution $(\psi_1,\psi_2)$ satisfying \eqref{reg (psi1,psi2)}, then considering $u_\pm$ the unique classical solution of $(P_\pm)$, we obtain that $u$ is the unique classical solution of (P).
\end{proof}

\subsection{Functional calculus}\label{sect funct calc}

In this section, by using functional calculus, we rewrite the operators defined in \eqref{U V +-}, \eqref{Pi+}, \eqref{Pi-} and \eqref{Qi-}, to inverse the determinant operator of system \eqref{syst trans gén} and system \eqref{syst trans M}. 

To this end, we recall some classical notations. For $\theta \in (0,\pi)$, we denote by $H(S_\theta)$ the space of holomorphic functions on $S_\theta$ (defined by \eqref{defsector}) with values in $\CC$. Moreover, we consider the following subspace of~$H(S_\theta)$:
$$
\mathcal{E}_\infty(S_\theta)\;:=\;\left\{f \in H(S_\theta) : f = O(|z|^{-s}) ~(|z| \rightarrow +\infty) \text{ for some } s > 0\right\}.
$$
In other words, $\mathcal{E}_\infty(S_\theta)$ is the space of polynomially decreasing holomorphic functions at $+\infty$. Let $T$ be an invertible sectorial operator of angle $\theta_T \in (0,\pi)$. If $f\in \mathcal{E}_\infty(S_\theta)$, with $\theta \in (\theta_T,\pi)$, then we can define, by functional calculus, $f(T) \in \L(X)$, see \cite{haase}, p. 45. 

Then, we recall a useful result from \cite{LMMT}, Lemma 5.3, p. 370.

\begin{Lem}[\cite{LMMT}]\label{Lem inv}
Let $P$ be an invertible sectorial operator in $X$ with angle $\theta$, for all $\theta \in (0, \pi)$. Let \mbox{$G \in H(S_{\theta})$}, for some $\theta \in (0, \pi)$, such that
\begin{enumerate}
\item[$(i)$] $1-G \in \mathcal{E}_\infty(S_{\theta})$,

\item[$(ii)$] $G (x)\neq 0$ for any $x \in \RR_+ \setminus \{0\}$.
\end{enumerate}
Then, $G(P) \in \L(X)$, is invertible with bounded inverse.
\end{Lem}

Let $r \in \RR$, $r_m = \max(-r,0)$, $\delta > 0$ and $z \in \CC \setminus (-\infty, r_m]$. We set 
$$\left \{ \begin{array}{l}
u_{\delta,r} (z) = \left\{\begin{array}{ll}
\dis 1 - e^{-\delta(\sqrt{z + r} + \sqrt{z})} - \frac{1}{r} (\sqrt{z + r} + \sqrt{z})^2 \left(e^{-\delta\sqrt{z}} - e^{-\delta\sqrt{z + r} }\right), & \text{if } r \in \RR \setminus\{0\} \\ \ecart
\dis 1 - e^{-2\delta \sqrt{z}} - 2 \delta \sqrt{z} e^{-\delta \sqrt{z}}, & \text{if } r = 0
\end{array}\right. \\ \\

v_{\delta,r} (z) = \left\{\begin{array}{ll}
\dis 1 - e^{-\delta(\sqrt{z + r} + \sqrt{z})} + \frac{1}{r} (\sqrt{z + r} + \sqrt{z})^2 \left(e^{-\delta\sqrt{z}} - e^{-\delta\sqrt{z + r} }\right), & \text{if } r \in \RR \setminus\{0\} \\ \ecart
\dis 1 - e^{-2\delta \sqrt{z}} + 2 \delta \sqrt{z} e^{-\delta \sqrt{z}}, & \text{if } r = 0
\end{array}\right.
\end{array}\right.$$
and when $u_{\delta,r}(z) \neq 0$, $v_{\delta,r}(z) \neq 0$, we note
\begin{equation*}
\begin{array}{l}
f_{\delta,r,1} (z) = \left\{\hspace{-0.14cm}\begin{array}{ll}
\dis ~~\, \left(\sqrt{z+r} + \sqrt{z}\right) \sqrt{z+r} \,u^{-1}_{\delta,r} (z) \left(1 + e^{-\delta\sqrt{z}}\right)\left(1 + e^{-\delta\sqrt{z+r}}\right) \\ \ecart
+ \dis \left(\sqrt{z+r} + \sqrt{z}\right) \sqrt{z+r} \,v^{-1}_{d,r} (z) \left(1 - e^{-\delta\sqrt{z}}\right)\left(1 - e^{-\delta\sqrt{z+r}} \right), & \text{if } r \in \RR \setminus\{0\} \\ \\
\dis \left(u^{-1}_{\delta,0} (z) + v^{-1}_{\delta,0} (z)\right) \left(1 - e^{-2\delta \sqrt{z}}\right), & \text{if } r = 0,
\end{array}\right.\end{array}
\end{equation*}
\begin{equation*}
\begin{array}{l}
f_{\delta,r,2} (z) = \left\{\begin{array}{ll}
\dis - \left(\sqrt{z+r} + \sqrt{z}\right) u^{-1}_{\delta,r}(z) \left(1 + e^{-\delta\sqrt{z}}\right)\left(1 - e^{-\delta\sqrt{z+r}}\right) \\ \ecart
- \dis \left(\sqrt{z+r} + \sqrt{z}\right) v^{-1}_{\delta,r} (z)\left(1 - e^{-\delta\sqrt{z}}\right)\left(1 + e^{-\delta\sqrt{z+r}}\right), & \text{if } r \in \RR \setminus\{0\} \\ \\

\dis u^{-1}_{\delta,0} (z) \left(1 - e^{-\delta \sqrt{z}}\right)^2 +  v^{-1}_{\delta,0} (z) \left(1 + e^{-\delta \sqrt{z}}\right)^2 , & \text{if } r = 0,
\end{array}\right.\end{array}
\end{equation*}
and
\begin{equation*}
\begin{array}{l}
f_{\delta,r,3} (z) = \left\{\begin{array}{ll}
\dis - \left(\sqrt{z+r} + \sqrt{z}\right) u^{-1}_{\delta,r} (z) \left(1 - e^{-\delta\sqrt{z}}\right) \left(1 - e^{-\delta\sqrt{z+r}}\right) \\ \ecart
- \dis \left(\sqrt{z+r} + \sqrt{z}\right) v^{-1}_{\delta,r} (z) \left(1 + e^{-\delta\sqrt{z}}\right) \left(1 + e^{-\delta\sqrt{z+r}}\right), & \text{if } r \in \RR \setminus\{0\} \\ \\

\dis u^{-1}_{\delta,0} (z) \left(1 + e^{-\delta \sqrt{z}}\right)^2 +  v^{-1}_{\delta,0} (z) \left(1 - e^{-\delta \sqrt{z}}\right)^2 , & \text{if } r = 0
\end{array}\right.\end{array}
\end{equation*}
\begin{Rem}
Note that, from $(H_2)$ and $(H_3)$, if $r_\pm \neq 0$, we have
$$\left\{\begin{array}{rrr}
u_{c,r_-} (-A) = U_-, & u_{d,r_+} (-A) = U_+, & v_{c,r_-} (-A) = V_-, \\
v_{d,r_+} (-A) = V_+, & k_- f_{c,r_-,1} (-A) = P_1^-, & k_+ f_{d,r_+,1} (-A) = P_1^+, \\
k_- f_{c,r_-,2} (-A) = P_2^-, & k_+ f_{d,r_+,2} (-A) = P_2^+, & k_- f_{c,r_-,3} (-A) = P_3^-, \\
k_+ f_{d,r_+,3} (-A) = P_3^+, 
\end{array}\right.$$
and if $r_-=0$, we obtain
$$\left\{\begin{array}{rrr}
u_{c,0} (-A) = U_-, & v_{c,0} (-A) = V_-, & k_- f_{c,0,1} (-A) = Q_1^-, \\
k_- f_{c,0,2} (-A) = Q_2^-,& k_- f_{c,0,3} (-A) = Q_3^-. 
\end{array}\right.$$
\end{Rem}
\begin{Rem}\label{Rem f1 +- > 0 et f2,f3 +- < 0}
Let $\delta > 0$, $r \in \RR$ and $x \in (r_m,+\infty)$. Then, when $r=0$, we have 
$$1 - e^{-2\delta \sqrt{x}} \pm 2 \delta \sqrt{x} e^{-\delta \sqrt{x}} = 2 e^{-\delta \sqrt{x}}\left(\sinh(\delta \sqrt{x}) \pm \delta \sqrt{x}\right) > 0,$$
and from \cite{LMMT}, Lemma 5.2, p. 369, it clear that $u_{\delta,r}(x)>0$ and $v_{\delta,r}(x) > 0$. Thus, when $r\neq 0$, we deduce that
$$f_{\delta,r,1}(x) > 0 \quad \text{and} \quad f_{\delta,r,2}(x),f_{\delta,r,3}(x) < 0,$$
and when $r=0$, we obtain
$$f_{\delta,0,1}(x), f_{\delta,0,2}(x), f_{\delta,0,3}(x) > 0.$$
\end{Rem}
Moreover, for $z \in \CC \setminus (-\infty, r_m]$ and $r \in \RR\setminus\{0\}$, we define
$$\begin{array}{lll}
g_{\delta,r}(z) &\hspace{-0.1cm} = & \hspace{-0.1cm}\dis -\sqrt{z + r} \left(\left(1 - e^{-2\delta(\sqrt{z + r} + \sqrt{z})}\right)^2 - \frac{1}{r^2} (\sqrt{z + r} + \sqrt{z})^4 \left(e^{-2\delta\sqrt{z}} - e^{-2\delta\sqrt{z + r}}\right)^2 \right) \\ \ecart

&&\hspace{-0.1cm} \dis + \sqrt{z} \left(\left(1 - e^{-\delta(\sqrt{z + r} + \sqrt{z})}\right)^2 + \frac{1}{r} (\sqrt{z + r} + \sqrt{z})^2 \left(e^{-\delta\sqrt{z}} - e^{-\delta\sqrt{z + r}}\right)^2 \right)^2,
\end{array}$$
and for $r=0$, we set
$$g_{\delta,0} (z) = \left(1 + \sqrt{z}\right) \left(1 - e^{-2\delta\sqrt{z}}\right)^4 + 4 \left(1 - e^{-2\delta \sqrt{z}}\right)^2 e^{-2\delta\sqrt{z}} - 16 \,\delta^2 z \,e^{-4\delta\sqrt{z}}$$
\begin{Lem} \label{Lem g < 0}
Let $\delta > 0$ and $x \in (r_m,+\infty)$. Thus, if $r \in \RR\setminus\{0\}$, then $g_{\delta,r}(x) < 0$ and if $r=0$, then $g_{\delta,0}(x) > 0$. 
\end{Lem}
\begin{proof}
Let $\delta> 0$ and $r \in \RR\setminus\{0\}$. For all $x \in (r_m,+\infty)$, from \cite{LLMT}, Lemma 4.4, p. 2950, we obtain the result.

Now, consider that $r=0$. Then
$$\begin{array}{lll}
g_{\delta,0} (x) & = & \dis \left(1 + \sqrt{x}\right) \left(1 - e^{-2\delta\sqrt{x}}\right)^4 + 4e^{-2\delta\sqrt{x}} \left( \left(1 - e^{-2\delta \sqrt{x}}\right)^2  - 4 \,\delta^2 x \,e^{-2\delta\sqrt{x}}\right) \\ \\

& = & \dis \left(1 + \sqrt{x}\right) \left(1 - e^{-2\delta\sqrt{x}}\right)^4 \\ \ecart
&& \dis + 4e^{-2\delta\sqrt{x}} \left( 1 - e^{-2\delta \sqrt{x}}  - 2 \,\delta \sqrt{x} \,e^{-\delta\sqrt{x}}\right)\left(1 - e^{-2\delta \sqrt{x}} + 2 \,\delta \sqrt{x}\, e^{-\delta\sqrt{x}}\right). 
\end{array}$$
Hence, since $\delta,x>0$, we have
$$\begin{array}{lll}
\dis 1 - e^{-2\delta \sqrt{x}}  - 2 \,\delta \sqrt{x} \,e^{-\delta\sqrt{x}} & = & \dis e^{-\delta\sqrt{x}}\left(e^{\delta\sqrt{x}} - e^{-\delta\sqrt{x}} -  2 \,\delta \sqrt{x}\right) \\ \ecart 
& = & \dis 2 e^{-\delta\sqrt{x}} \left(\sinh(\delta\sqrt{x}) - \delta \sqrt{x}\right) > 0.
\end{array}$$
Finally, we deduce that $g_{\delta,0} > 0$.
\end{proof}

\section{Proof of the main results}

In both cases, assume that problem (P) has a unique classical solution; thus, from \refT{Th syst trans}, respectively \refT{Th syst trans M}, \eqref{reg phi +-} holds. Conversely, assume that \eqref{reg phi +-} holds, then due to \refT{Th syst trans}, respectively \refT{Th syst trans M}, we have to prove that system \eqref{syst trans gén}, respectively system \eqref{syst trans M}, has a unique solution such that \eqref{reg (psi1,psi2)} holds. 

The proof is divided in three parts for both cases. First, we will make explicit, in the first case, the determinant of system \eqref{syst trans gén} and in the second case, the determinant of system \eqref{syst trans M}. Then, in the two cases, we will show the uniqueness of the solution. To this end, we will invert the determinant thanks to functional calculus. Finally, we will prove, in all cases, that $\psi_1$ and $\psi_2$ have the expected regularity. 

\subsection{Calculus of the determinant}

\subsubsection{First case}

Here, we consider $r_+,r_- \in \RR \setminus\{0\}$. We have to make explicit the determinant of system \eqref{syst trans gén} that we recall here
$$\left\{\begin{array}{cclll}
\left(P_1^- - P_1^+ \right) M \psi_1 &+& \left(P_2^+ + P_2^-\right) \psi_2 &=& S_1 \\ \ecart
\left(P_3^+ + P_3^- \right) \psi_1 &+& \left(P_1^- - P_1^+ \right) \psi_2 &=& S_2.
\end{array}\right.$$
We write the previous system as a matrix equation $\Lambda_1 \Psi = S$, where
$$\Lambda_1 = \left(\begin{array}{cc}
\left(P_1^- - P_1^+ \right) M & \left(P_2^+ + P_2^-\right) \\
\left(P_3^+ + P_3^- \right) & \left(P_1^- - P_1^+ \right)
\end{array}\right), \quad \Psi = \left(\begin{array}{c}
\psi_1 \\
\psi_2 
\end{array}\right) \quad \text{and} \quad S = \left(\begin{array}{c}
S_1 \\
S_2
\end{array}\right).$$
To solve system \eqref{syst trans gén}, we will study the determinant
$$\det(\Lambda_1) := M\left(P_1^- - P_1^+ \right)^2 - \left(P_2^+ + P_2^- \right)\left(P_3^+ + P_3^- \right),$$
of the matrix $\Lambda_1$. Thus, we set
\begin{equation} \label{det lambda 1}
\det(\Lambda_1) = D_1^+ + D_1^- + D_2,
\end{equation}
where
$$\left\{\begin{array}{lll}
D_1^+ & = & \dis M \left(P_1^+\right)^2 - P_3^+ P_2^+ \\ \ecart
D_1^- & = & \dis M \left(P_1^-\right)^2 - P_3^- P_2^- \\ \ecart
D_2 & = & \dis - P_3^+ P_2^- - P_3^- P_2^+ - 2 M P_1^+ P_1^-.
\end{array}\right.$$
Then, we recall the result of \cite{LLMT} (Lemma 5.1, p. 2953), describing the determinant.
\begin{Lem}[\cite{LLMT}]\label{Lem D1+ D1-}
We have
\begin{enumerate}
\item $D_1^+ = -4 k_+^2 (L_+ + M)^2 U_+^{-2} V_+^{-2} D^+$, with 
$$\begin{array}{rcl}
D^+ & = & \dis L_+ \left(\left(I - e^{2d(L_+ + M)}\right)^2 - \frac{1}{r_+^2} (L_+ + M)^4 \left(e^{2dM} - e^{2dL_+}\right)^2\right) \\ \ecart

&& \dis - M \left(\left(I - e^{d(L_+ + M)}\right)^2 + \frac{1}{r_+^2} (L_+ + M)^2 \left(e^{dM} - e^{dL_+}\right)^2\right)^2.
\end{array}$$

\item $D_1^- = -4 k_-^2 (L_- + M)^2 U_-^{-2} V_-^{-2} D^-$, with 
$$\begin{array}{rcl}
D^- & = & \dis L_- \left(\left(I - e^{2c(L_- + M)}\right)^2 - \frac{1}{r_-^2} (L_- + M)^4 \left(e^{2cM} - e^{2cL_-}\right)^2\right) \\ \ecart

&& \dis - M \left(\left(I - e^{c(L_- + M)}\right)^2 + \frac{1}{r_-^2} (L_- + M)^2 \left(e^{cM} - e^{cL_-}\right)^2\right)^2.
\end{array}$$
\end{enumerate}
\end{Lem}

\subsubsection{Second case}

Here, we consider $r_+ \in \RR \setminus\{0\}$ and $r_- = 0$. As previously, we make explicit the determinant of system \eqref{syst trans M} that we recall here
$$\left\{\begin{array}{cclll}
\left(P_1^+ - 2M Q_1^- \right) M \psi_1 &-& \left(P_2^+ + 2M Q_2^-\right) \psi_2 &=& S_3 \\ \ecart
\left(P_3^+ + 2M Q_3^- \right) \psi_1 &+& \left(2M Q_1^- - P_1^+ \right) \psi_2 &=& S_4,
\end{array}\right.$$
We write this system as a matrix equation $\Lambda_2 \Psi = \tilde{S}$, where
$$\Lambda_2 = \left(\begin{array}{cc}
\left(P_1^+ - 2M Q_1^- \right) M & -\left(P_2^+ + 2M Q_2^-\right) \\
\left(P_3^+ + 2M Q_3^- \right) & \left(2M Q_1^- - P_1^+ \right)
\end{array}\right), \quad \Psi = \left(\begin{array}{c}
\psi_1 \\
\psi_2 
\end{array}\right) \quad \text{and} \quad \tilde{S} = \left(\begin{array}{c}
S_3 \\
S_4
\end{array}\right).$$
To solve system \eqref{syst trans gén}, we will study the determinant
$$\det(\Lambda_2) := - M \left(P_1^+ - 2M Q_1^- \right)^2 + \left(P_3^+ + 2M Q_3^- \right) \left(P_2^+ + 2M Q_2^-\right),$$
of the matrix $\Lambda_2$. We set
\begin{equation} \label{det lambda 2}
\det(\Lambda_2) = D_3^+ + D_3^- + D_4,
\end{equation}
where
$$D_3^+ = P_2^+ P_3^+ - M \left(P_1^+\right)^2 \quad \text{and} \quad  D_3^- = 4 M^2 \left(Q_2^-Q_3^- - M \left(Q_1^-\right)^2\right),$$
with 
$$D_4 = 2M\left(P_3^+ Q_2^- + P_2^+ Q_3^- + M P_1^+ Q_1^-\right).$$
\begin{Lem}\label{Lem D3+ D3-}
We have
\begin{enumerate}
\item $D_3^+ = k_+^2 (L_+ + M)^2 U_+^{-2} V_+^{-2} D_0^+$, with 
$$\begin{array}{rcl}
D_0^+ & = & \dis L_+ \left(\left(I - e^{2d(L_+ + M)}\right)^2 - \frac{1}{r_+^2} (L_+ + M)^4 \left(e^{2dM} - e^{2dL_+}\right)^2\right) \\ \ecart

&& \dis - M \left(\left(I - e^{d(L_+ + M)}\right)^2 + \frac{1}{r_+^2} (L_+ + M)^2 \left(e^{dM} - e^{dL_+}\right)^2\right)^2.
\end{array}$$

\item $D_3^- =  16 k_-^2 M^2 U_-^{-2} V_-^{-2} D_0^-$, with 
$$D_0^- = \left(I-M\right) \left(I - e^{2cM}\right)^4 + 4 \left(I - e^{2c M}\right)^2 e^{2cM} - 16 c^2 M^2 e^{4cM}.$$
\end{enumerate}
\end{Lem}

\begin{proof}\hfill

\begin{enumerate}
\item We have 
$$P_2^+ P_3^+ = k_+^2 (L_+ + M)^2 L_+ U^{-2}_+ V^{-2}_+ D'_+,$$ 
where
$$\begin{array}{rcl}
D'_+ & = & \dis \left(U^2_+ + V^2_+\right) (I - e^{2dM})(I - e^{2dL_+}) \\ \ecart

&& + U_+ V_+ \left((I + e^{dM})^2(I + e^{dL_+})^2 + (I - e^{dM})^2(I - e^{dL_+})^2\right) \\ \\

& = & \dis \left(U_+^2 + V_+^2 \right) \left[\left(I + e^{d(L_+ + M)}\right)^2 - \left(e^{dM} + e^{dL_+}\right)^2\right] \\ \ecart
&& \dis + 2 U_+ V_+ \left[\left(I + e^{d(L_+ + M)}\right)^2 + \left(e^{dM} + e^{dL_+}\right)^2\right] \\ \\

& = & \dis \left(U_+ + V_+ \right)^2 \left(I + e^{d(L_+ + M)}\right)^2 - \left(V_+ - U_+\right)^2 \left(e^{dM} + e^{dL_+}\right)^2.
\end{array}$$
Moreover, from \eqref{U V +-}, we obtain that
\begin{equation} \label{U+V et U-V}
U_+ + V_+ = 2\left(I - e^{d(L_+ + M)}\right) \quad \text{and} \quad V_+ - U_+ = \frac{2}{r_+} (L_+ + M)^2 \left(e^{dM} - e^{dL_+}\right).
\end{equation}
Then
$$\begin{array}{rcl}
D'_+ & = & \dis 4\left(I - e^{d(L_+ + M)}\right)^2 \left(I + e^{d(L_+ + M)}\right)^2 \\ \ecart
&& \dis - \frac{4}{r_+^2} (L_+ + M)^4 \left(e^{dM} - e^{dL_+}\right)^2 \left(e^{dM} + e^{dL_+}\right)^2 \\ \\

& = & \dis 4\left(I - e^{2d(L_+ + M)}\right)^2 - \frac{4}{r_+^2} (L_+ + M)^4 \left(e^{2dM} - e^{2dL_+}\right)^2.
\end{array}$$
Furthermore, we have 
$$M \left(P_1^+\right)^2 = k_+^2 (L_+ + M)^2 M U_+^{-2} V_+^{-2} D''_+,$$
where
$$\begin{array}{rcl}
D''_+ & = & \dis \left(V_+ (I + e^{dM})(I - e^{dL_+}) +  U_+ (I - e^{dM})(I+e^{dL_+}) \right)^2 \\ \ecart

& = & \dis \left[\left(U_+ + V_+\right)\left(I - e^{d(L_+ + M)}\right) + \left(V_+ - U_+\right)\left(e^{dM} - e^{dL_+}\right)\right]^2, 
\end{array}$$ 
and due to \eqref{U+V et U-V}, it follows that
$$D''_+ = \left[2 \left(I - e^{d(L_+ + M)}\right)^2 + \frac{2}{r_+} (L_+ + M)^2 \left(e^{dM} - e^{dL_+}\right)^2\right]^2.$$
Finally, we deduce that
$$\begin{array}{rcl}
D_3^+ & = & \dis P_2^+ P_3^+ - M \left(P_1^+\right)^2 \\ \ecart

& = & \dis k_+^2 (L_+ + M)^2 U^{-2}_+ V^{-2}_+ \left(L_+ D'_+ - M D''_+\right),
\end{array}$$
and setting $D_0^+ = L_+ D'_+ - M D''_+$, we obtain the expected result.

\item We have
$$Q_2^-Q_3^- = k_-^2 U_-^{-2}V_-^{-2}D'_-,$$
where
$$\begin{array}{rcl}
D'_- & = & \dis \left(V_-\left(I - e^{cM}\right)^2 + U_- \left(I + e^{cM}\right)^2 \right) \left(V_- \left(I + e^{cM}\right)^2 + U_- \left(I - e^{cM}\right)^2\right) \\ \ecart

& = & \dis \left(U_-^2 + V_-^2 \right) \left(I - e^{2cM}\right)^2 + 2 U_- V_- \left(I - e^{2cM}\right)^2 + 16 U_- V_- e^{2cM}, 
\end{array}$$
and
$$\begin{array}{rcl}
M \left(Q_1^-\right)^2 & = & \dis k_-^2 M U_-^{-2} V_-^{-2} \left(U_- + V_- \right)^2  \left(I - e^{2cM}\right)^2 \\ \ecart

& = & \dis k_-^2 U_-^{-2} V_-^{-2} \left[M\left(U_-^2 + V_-^2 \right)  \left(I - e^{2cM}\right)^2 + 2 M U_- V_- \left(I - e^{2cM}\right)^2\right].
\end{array}$$
Thus
$$Q_2^-Q_3^- - M \left(Q_1^-\right)^2 = k_-^2 U_-^{-2} V_-^{-2} D''_-,$$
where 
$$\begin{array}{rcl}
D''_- & = & \dis \left(I-M\right) \left(U_-^2 + V_-^2 \right)  \left(I - e^{2cM}\right)^2 + 2 \left(I-M\right) U_- V_- \left(I - e^{2cM}\right)^2 \\ \ecart
&& \dis + 16 U_- V_- e^{2cM} \\ \\

& = & \dis \left(I-M\right) \left(U_- + V_- \right)^2  \left(I - e^{2cM}\right)^2 + 16 U_- V_- e^{2cM}.
\end{array}$$
Moreover, from \eqref{U V +-}, we obtain that
$$U_- + V_- = 2 \left(I - e^{2c M}\right) \quad \text{and} \quad U_-V_- = \left(I - e^{2c M}\right)^2 - 4 c^2 M^2 e^{2cM}.$$
Then
$$D''_- = 4 \left(I-M\right) \left(I - e^{2cM}\right)^4 + 16 \left(I - e^{2c M}\right)^2 e^{2cM} - 64 c^2 M^2 e^{4cM}.$$
Therefore, it follows that
$$\begin{array}{rcl}
D_3^- & = & \dis 4 M^2 \left(Q_2^-Q_3^- - M \left(Q_1^-\right)^2\right) \\ \ecart

& = & \dis 16 k_-^2 M^2 U_-^{-2} V_-^{-2} D_0^-,
\end{array}$$
where $D_0^- = \dis\frac{1}{4}\,D''_-$.
\end{enumerate}
\end{proof}

\subsection{Inversion of the determinant}

\subsubsection{First case}

Here, we consider $r_+,r_- \in \RR \setminus\{0\}$. Let $r = \max(-r_+,-r_-,0) \geqslant 0$. By using functional calculus, we prove that the determinant of system \eqref{syst trans gén}, given by \eqref{det lambda 1}, is invertible with bounded inverse. Due to \refL{Lem D1+ D1-} and the definition of $D_2$, we obtain:
$$D_1^+ = g_1^+(-A), \quad D_1^- = g_1^-(-A) \quad \text{and} \quad D_2 = g_2(-A),$$ 
where, for $z \in \CC \setminus \RR_-$, we have set
$$\left\{ \begin{array}{lcl}
g_1^+ (z) & = & \dis 4 k_+^2 (\sqrt{z + r_+} + \sqrt{z})^2 u_{d,r_+}^{-2}(z) v_{d,r_+}^{-2}(z) g_{d,r_+}(z) \\ \ecart

g_1^- (z) & = & \dis 4 k_-^2 (\sqrt{z + r_-} + \sqrt{z})^2 u_{c,r_-}^{-2}(z) v_{c,r_-}^{-2}(z) g_{c,r_-}(z) \\ \ecart

g_2(z) & = & \dis k_+ f_{d,r_+,1} (z) k_- f_{c,r_-,3}(z) + k_- f_{c,r_-,1}(z) k_+ f_{d,r_+,3}(z)\vspace{0.1cm}\\
&&\dis - 2 \sqrt{z}\, k_+ f_{d,r_+,2}(z) k_- f_{c,r_-,2}(z),
\end{array}\right.$$
with $u_{\delta,r}$, $v_{\delta,r}$, $g_{\delta,r}$ and $f_{\delta,r,i}$ the complex functions defined in section~\ref{sect funct calc}. Thus
\begin{equation} \label{f(-A)}
\det(\Lambda_1) = D_1^+ + D_1^- + D_2 = f_1(-A),
\end{equation}
with $f_1 = g_1^+ + g_1^- + g_2$. Note that, for some $\theta \in (0,\pi)$, we have $f_1 \in H(S_\theta)$ and due to \refR{Rem f1 +- > 0 et f2,f3 +- < 0} and \refL{Lem g < 0}, for $x> 0$, we have 
\begin{equation} \label{f < 0}
f_1(x)= g_1^+ (x) + g_1^- (x) + g_2 (x) < 0.
\end{equation}
Let $C_1,C_2$ be two linear operators in $X$. We denote by $C_1 \sim C_2$ the equality $C_1 = C_2 + \Sigma$, where $\Sigma$ is a finite sum of terms of type $k L_+^l L_-^m M^n e^{\alpha L_+} e^{\beta L_-} e^{\delta M}$, where $k \in \RR$; $l,m,n \in \NN$; $\alpha ,\beta ,\delta \in \RR_+$ with $\alpha + \beta + \delta \neq 0$. Note that $\Sigma$ is a regular term in the sense: 
$$\Sigma \in \L(X) \quad\text{with}\quad \Sigma(X) \subset D(M^\infty):= \bigcap_{k\geqslant0}D(M^k).$$
Since we have $U_\pm \sim I, V_\pm \sim I$, then by setting $W=U_- U_+ V_- V_+ \sim I$, we deduce that
$$\left\{\begin{array}{lcllcl}
W P_1^+ & \sim & 2k_+ (L_+ + M), & W P_1^- & \sim & 2k_- (L_- + M) \\ \ecart
 
W P_2^+ & \sim & 2k_+ (L_+ + M), & W P_2^- & \sim & 2k_- (L_- + M) \\ \ecart
 
W P_3^+ & \sim & 2k_+ (L_+ + M) L_+, & W P_3^- & \sim & 2k_- (L_- + M) L_-.
\end{array}\right.$$
Thus
$$\begin{array}{rcl}
W^2 \det(\Lambda_1) & = & \dis M \left(W P_1^+\right)^2 -  \left(W P_2^+ W P_3^+\right) + M \left(W P_1^-\right)^2 - \left(W P_2^- W P_3^-\right) \\ \ecart
&& \dis - \left(W P_2^- W P_3^+ + W P_2^+ W P_3^- + 2 M W P_1^+ W P_1^- \right) \\ \\

& \sim & \dis -4 k_+^2 (L_+ + M)^2 (L_+ - M) - 4 k_-^2 (L_- + M)^2 (L_- - M) \\ \ecart
&& \dis - 4 k_+ k_- (L_+ + M) (L_- + M) (L_+ + L_- + 2M).
\end{array}$$
From \eqref{L-M=}, we have
$$\begin{array}{rcl}
-W^2 \det(\Lambda_1) & \sim & \dis 4 k_+^2 r_+ (L_+ + M) + 4 k_-^2 r_- (L_- + M) \\ \ecart
&& \dis + 4 k_+ k_- (L_+ + M) (L_- + M) (L_+ + L_- + 2M) \\ \\

& \sim & \dis 4 k_+ l_+ (L_+ + M) + 4 k_- l_- (L_- + M) \\ \ecart
&& \dis + 4 k_+ k_- (L_+ + M) (L_- + M) (L_+ + L_- + 2M)
\end{array}$$
Hence, we note
$$B_1 = 4 k_+ l_+ (L_+ + M) + 4 k_- l_- (L_- + M) + 4 k_+k_- (L_+ + M)(L_- + M)(L_+ + L_- + 2M).$$
Thus, we obtain 
\begin{equation}\label{det lambda =  W-2 BF}
\det(\Lambda_1) = -W^{-2} \left(B_1 + \sum_{j \in J} k_j L_+^{l_j} L_-^{m_j} M^{n_j} e^{\alpha_j L_+} e^{\beta_j L_-} e^{\delta_j M}\right),
\end{equation}
where $J$ is a finite set and for any $j \in J$:
$$k_j \in \RR;~l_j, m_j, n_j \in \NN,~ \alpha_j, \beta_j, \delta_j \in \RR_+ \quad \text{with} \quad \alpha_j + \beta_j + \delta_j \neq 0.$$
We set 
$$B_2 = I + \frac{l_+}{k_-} (L_- + M)^{-1} (L_+ + L_- + 2M)^{-1} +  \frac{l_-}{k_+} (L_+ + M)^{-1} (L_+ + L_- + 2M)^{-1}$$
such that 
$$B_1 = 4 k_+k_- (L_+ + M)(L_- + M)(L_+ + L_- + 2M)B_2.$$

\begin{Prop}\label{Prop B1 inv}
Assume that $(H_1)$, $(H_2)$, $(H_3)$, $(H_4)$ hold and $k_+k_- >0$. If one of the following assumptions holds
\begin{itemize}
\item $\dis\frac{l_+}{k_-} > 0$ and $\dis\frac{l_-}{k_+} > 0$,

\item $\dis\frac{l_+}{k_-} < 0$ and $\dis\frac{l_-}{k_+} < 0$, such that 
\begin{equation}\label{hyp - -}
(l_+ - l_-) (k_+ - k_-) \geqslant 0,
\end{equation}

\item $\dis\frac{l_+}{k_-} > 0$ and $\dis\frac{l_-}{k_+} < 0$, such that 
\begin{equation}\label{hyp + -}
-6 l_-k_+ + l_+k_+ + l_-k_- \geqslant 0,
\end{equation}

\item $\dis\frac{l_+}{k_-} < 0$ and $\dis\frac{l_-}{k_+} > 0$, such that 
\begin{equation}\label{hyp - +}
-6 l_+k_- + l_+k_+ + l_-k_- \geqslant 0,
\end{equation}
\end{itemize}
then, $b_2(x) > 0$, for $x>r \geqslant 0$ and operator $B_1$, defined above, is invertible with bounded inverse.
\end{Prop} 
\begin{Rem}\label{Rem équivalences}
Since $k_+k_- > 0$, then we have the following equivalences
$$\frac{l_+}{k_-} > 0 \Longleftrightarrow r_+ > 0 \quad \text{and} \quad \frac{l_-}{k_+} > 0 \Longleftrightarrow r_- > 0.$$
\end{Rem}
\begin{proof}
From $(H_2)$ and $(H_3)$, since $k_+k_- \neq 0$, it is clear that 
$$0 \in \rho\left(4 k_+k_- (L_+ + M)(L_- + M)(L_+ + L_- + 2M)\right).$$
Thus, it remains to prove that $0 \in \rho(B_2)$. To this end, we use \refL{Lem inv}.

Let $z \in \CC \setminus (-\infty, r]$. We set
\begin{equation}\label{b2}
\begin{array}{lll}
b_2(z) & = & \dis 1 + \frac{l_+}{k_-} \frac{1}{(\sqrt{z+r_-} + \sqrt{z})(\sqrt{z+r_+} + \sqrt{z+r_-} + 2\sqrt{z})} \\ \ecart
&& \dis +  \frac{l_-}{k_+} \frac{1}{(\sqrt{z+r_+} + \sqrt{z})(\sqrt{z+r_+} + \sqrt{z+r_-} + 2\sqrt{z})},
\end{array}
\end{equation}
hence $b_2(-A) = B_2$. Then, for all $x > r\geqslant 0$, it follows
$$\begin{array}{lll}
b_2(x) & = & \dis 1 + \frac{l_+}{k_-} \frac{1}{(\sqrt{x+r_-} + \sqrt{x})(\sqrt{x+r_+} + \sqrt{x+r_-} + 2\sqrt{x})} \\ \ecart
&& \dis +  \frac{l_-}{k_+} \frac{1}{(\sqrt{x+r_+} + \sqrt{x})(\sqrt{x+r_+} + \sqrt{x+r_-} + 2\sqrt{x})}
\end{array}$$
Our aim is to prove that $b_2(x) > 0$, for all $x > r$. To this end, we set
$$y = x - r > 0,$$
hence
$$\begin{array}{lll}
b_2(y+r) & = &  \dis 1 + \frac{l_+}{k_-} \frac{1}{(\sqrt{y+r+r_-} + \sqrt{y+r})(\sqrt{y+r+r_+} + \sqrt{y+r+r_-} + 2\sqrt{y+r})} \\ \ecart
&& \dis +  \frac{l_-}{k_+} \frac{1}{(\sqrt{y+r+r_+} + \sqrt{y+r})(\sqrt{y+r+r_+} + \sqrt{y+r+r_-} + 2\sqrt{y+r})} \\ \\

& = & \dis 1 + \frac{1}{(\sqrt{y+r+r_+} + \sqrt{y+r+r_-} + 2\sqrt{y+r})} b_3(y),
\end{array}$$
where
$$b_3(y) = \frac{\frac{l_+}{k_-}}{(\sqrt{y+r+r_-} + \sqrt{y+r})} + \frac{\frac{l_-}{k_+}}{(\sqrt{y+r+r_+} + \sqrt{y+r})}.$$
Then
$$\begin{array}{lll}
b_3'(y) & = & \dis \frac{-\frac{l_+}{k_-} \left(\frac{1}{2\sqrt{y+r+r_-}} + \frac{1}{2\sqrt{y+r}} \right)}{(\sqrt{y+r+r_-} + \sqrt{y+r})^2} + \frac{-\frac{l_-}{k_+}\left(\frac{1}{2\sqrt{y+r+r_+}} + \frac{1}{2\sqrt{y+r}}\right)}{(\sqrt{y+r+r_+} + \sqrt{y+r})^2}
\end{array}$$
and
$$\begin{array}{lll}
b_2'(y+r) & = & \dis \frac{-\left(\frac{1}{2\sqrt{y+r+r_+}} + \frac{1}{2\sqrt{y+r+r_-}} + \frac{1}{\sqrt{y+r}}\right)}{(\sqrt{y+r+r_+} + \sqrt{y+r+r_-} + 2\sqrt{y+r})^2} b_3(y)\\ \ecart
&&\dis + \frac{1}{(\sqrt{y+r+r_+} + \sqrt{y+r+r_-} + 2\sqrt{y+r})} b_3'(y),
\end{array}$$
Now, we have to study the following fourth cases.
\begin{enumerate}
\item If $\dis\frac{l_+}{k_-} > 0$ and $\dis\frac{l_-}{k_+} > 0$, then it is clear that $b_3 > 0$ and $b_2>0$.

\item If $\dis\frac{l_+}{k_-} < 0$ and $\dis\frac{l_-}{k_+} < 0$, then $b_3' > 0$ and $b_2'>0$. Thus $b_2 (y+r) > b_2 (r)$ where
$$b_2(r) = 1 + \frac{1}{(\sqrt{r+r_+} + 2\sqrt{r})} b_3(0) > 1 + \frac{1}{2\sqrt{r}} b_3(0),$$
with
$$b_3(0) = \frac{\frac{l_+}{k_-}}{\sqrt{r + r_-} + \sqrt{r}} + \frac{\frac{l_-}{k_+}}{\sqrt{r+r_+} + \sqrt{r}}.$$
Since $\sqrt{r+r_+} >0$ and $\sqrt{r+r_-} >0$, it follows that 
$$\frac{\frac{l_+}{k_-}}{\sqrt{r + r_-} + \sqrt{r}} > \frac{\frac{l_+}{k_-}}{\sqrt{r}}  \quad \text{and} \quad \frac{\frac{l_-}{k_+}}{\sqrt{r+r_+} + \sqrt{r}} > \frac{\frac{l_-}{k_+}}{\sqrt{r}},$$
hence
$$b_3(0) > \frac{1}{\sqrt{r}} \left(\frac{l_+}{k_-} + \frac{l_-}{k_+}\right).$$
Thus, we obtain
$$b_2(r) > 1 + \frac{1}{2r} \left(\frac{l_+}{k_-} + \frac{l_-}{k_+}\right).$$
Moreover, we have
$$1 + \frac{1}{2r} \left(\frac{l_+}{k_-} + \frac{l_-}{k_+}\right) \geqslant 0 \Longleftrightarrow 2r \geqslant - \left(\frac{l_+}{k_-} + \frac{l_-}{k_+}\right),$$
where
$$-\left(\frac{l_+}{k_-} + \frac{l_-}{k_+}\right) = \left\{\begin{array}{ll}
\dis - r_- \left(\frac{l_+}{l_-} + \frac{k_-}{k_+}\right), & \text{if } r = -r_- \\ \ecart
\dis - r_+ \left(\frac{k_+}{k_-} + \frac{l_-}{l_+}\right), & \text{if } r = -r_+.
\end{array}\right.$$
Thus, we obtain that
$$2r \geqslant - \left(\frac{l_+}{k_-} + \frac{l_-}{k_+}\right) \Longleftrightarrow \left\{\begin{array}{ll}
\dis l_+k_+ + l_-k_- - 2 l_-k_+ \geqslant 0, & \text{if } r = -r_- \\ \ecart
\dis l_+k_+ + l_-k_- - 2 l_+k_- \geqslant 0, & \text{if } r = -r_+.
\end{array}\right.$$
Furthermore, since $k_+k_->0$, if $r = -r_-$, then $\dis -\frac{l_-}{k_-} \geqslant -\frac{l_+}{k_+}$, hence $\dis -l_- k_+ \geqslant - l_+k_-$ and if $r = -r_+$, then $\dis -\frac{l_+}{k_+} \geqslant -\frac{l_-}{k_-}$, hence $\dis - l_+ k_- \geqslant -l_-k_+$. It follows that
$$\left\{\begin{array}{ll}
\dis l_+k_+ + l_-k_- - 2 l_-k_+ \geqslant l_+k_+ + l_-k_- - l_+k_- - l_-k_+, & \text{if } r = -r_- \\ \ecart
\dis l_+k_+ + l_-k_- - 2 l_+k_- \geqslant l_+k_+ + l_-k_- - l_+k_- - l_-k_+, & \text{if } r = -r_+.
\end{array}\right.$$
Finally, if \eqref{hyp - -} holds, then we obtain $b_2 > 0$.

\item If $\dis\frac{l_+}{k_-} > 0$ and $\dis\frac{l_-}{k_+} < 0$, then since $k_+k_- > 0$, we have
$$\dis\frac{l_+}{k_-} > 0 \Longleftrightarrow \frac{l_+}{k_+}\frac{k_+}{k_-}k_-^2 > 0 \Longleftrightarrow r_+ > 0 \quad  \text{and} \quad \dis\frac{l_-}{k_+} < 0 \Longleftrightarrow \frac{l_-}{k_-}\frac{k_-}{k_+}k_+^2 < 0 \Longleftrightarrow r_- < 0.$$
Thus $r=-r_- > 0$, $r_+>0$ and
$$b_3(y) = \frac{\frac{l_+}{k_-}}{(\sqrt{y} + \sqrt{y+r})} + \frac{\frac{l_-}{k_+}}{(\sqrt{y+r+r_+} + \sqrt{y+r})}.$$
Since $\dis\frac{l_+}{k_-} > 0$ and $\sqrt{y} < \sqrt{y+r}$ it follows that
$$\frac{\frac{l_+}{k_-}}{\sqrt{y} + \sqrt{y+r}} > \frac{\frac{l_+}{k_-}}{2\sqrt{y+r}}.$$
In the same way, since $\dis\frac{l_-}{k_+} < 0$ and $\sqrt{y+r+r_+} > \sqrt{y+r}$, we deduce that
$$\frac{\frac{l_-}{k_+}}{\sqrt{y+r+r_+} + \sqrt{y+r}} > \frac{\frac{l_-}{k_+}}{2\sqrt{y+r}},$$
hence
$$b_3(y) > \frac{1}{2 \sqrt{y+r}} \left(\frac{l_+}{k_-} + \frac{l_-}{k_+}\right).$$
If $\dis\frac{l_+}{k_-} + \frac{l_-}{k_+} > 0$, then $b_3 > 0$ and $b_2 > 0$. If $\dis\frac{l_+}{k_-} + \frac{l_-}{k_+} < 0$, then we have
$$\begin{array}{lll}
b_2(y+r) & = & \dis 1 + \frac{1}{(\sqrt{y+r+r_+} + \sqrt{y} + 2\sqrt{y+r})} b_3(y) \\ \ecart
& > & \dis 1 + \frac{1}{3 \sqrt{y+r}} b_3(y) \\ \ecart
& > & \dis 1 + \frac{1}{6(y+r)}  \left(\frac{l_+}{k_-} + \frac{l_-}{k_+}\right).
\end{array}$$
Moreover, we have
$$1 + \frac{1}{6(y+r)}  \left(\frac{l_+}{k_-} + \frac{l_-}{k_+}\right) \geqslant 0 \Longleftrightarrow 6(y+r) + \frac{l_+}{k_-} + \frac{l_-}{k_+} \geqslant 0.$$
It is obvious that 
$$6(y+r) + \left(\frac{l_+}{k_-} + \frac{l_-}{k_+}\right) \geqslant 6r + \frac{l_+}{k_-} + \frac{l_-}{k_+},$$
thus, since $k_+k_- > 0$ and here $\dis r = -r_- = -\frac{l_-}{k_-}$, we deduce that the previous inequality becomes
$$-6\frac{l_-}{k_-} + \frac{l_+}{k_-} + \frac{l_-}{k_+} \geqslant 0 \Longleftrightarrow -6 l_-k_+ + l_+k_+ + l_-k_- \geqslant 0.$$
Finally, since $k_+k_- > 0$, if \eqref{hyp + -} holds, then $b_2 > 0$.

\item If $\dis\frac{l_+}{k_-} < 0$ and $\dis\frac{l_-}{k_+} > 0$, then here $r=-r_+$ and in the same way than previously, if \eqref{hyp - +} holds, then $b_2 > 0$.
\end{enumerate}
Since $r = \max(-r_+,-r_-,0) \geqslant 0$ and due to $(H_2)$ and $(H_3)$, we deduce that operator $-A - rI \in$ BIP$(X,\theta_A)$ with $0\in \rho(-A-rI)$.  Thus, considering $\tilde{b_2}(z) = b_2(z+r)$, with $z+r \in \CC \setminus \RR_-$, it follows that $\tilde{b_2}(-A - rI) = B_2$. Moreover, for a given $\theta \in (0,\pi)$, it is clear that $1-b_2, 1-\tilde{b_2} \in \mathcal{E}_\infty$. Finally, applying \refL{Lem inv} with $G = \tilde{b_2}$ and $P=-A-rI$, we deduce the result.
\end{proof}
Due to \eqref{det lambda =  W-2 BF} and \refP{Prop B1 inv}, it follows that 
\begin{equation}\label{det}
\det(\Lambda_1) = -W^{-2} B_1 F_1,
\end{equation}
where
\begin{equation}\label{F1}
F_1 = I + \sum_{j \in J} k_j B_1^{-1} L_+^{l_j} L_-^{m_j} M^{n_j} e^{\alpha_j L_+} e^{\beta_j L_-} e^{\delta_j M}.
\end{equation}
For $z \in \CC \setminus (-\infty,r]$, we set
\begin{equation}\label{b1}
b_1(z) = -4 k_+k_- (\sqrt{z + r_+} + \sqrt{z})(\sqrt{z+r_-} + \sqrt{z})(\sqrt{z+r_+} + \sqrt{z+r_-} + 2\sqrt{z})b_2(z),
\end{equation}
where $b_2$ is given by \eqref{b2} and
$$\tilde{f_1}(z) = 1 + \sum_{j \in J} k_j b_1(z)^{-1} \left(-\sqrt{z + r_+}\right)^{l_j} \left(-\sqrt{z + r_-}\right)^{m_j} \left(-\sqrt{z}\right)^{n_j} e^{-\alpha_j \sqrt{z + r_+}} e^{-\beta_j \sqrt{z + r_-}} e^{-\delta_j \sqrt{z}}.$$
Then, due to $(H_2)$ and $(H_3)$, we have $B_1 = b_1(-A)$ and $F_1 = \tilde{f_1}(-A)$. Moreover, from \eqref{f(-A)} and \eqref{det}, we obtain
$$f_1(-A) = \det(\Lambda_1) = -W^{-2} B_1 \tilde{f_1}(-A).$$
Note that, we have
\begin{equation}\label{f(z)}
f_1(z) = - u_{d,r_+}^{-2}(z)v_{d,r_+}^{-2}(z)u_{c,r_-}^{-2}(z)v_{c,r_-}^{-2}(z)b_1(z) \tilde{f_1}(z).
\end{equation}
\begin{Prop}\label{prop inv F}
Assume that $(H_1)$, $(H_2)$, $(H_3)$, $(H_4)$ hold and $k_+k_- >0$. Thus
\begin{itemize}
\item if $r_+ > 0$ and $r_- > 0$,

\item if $r_+ < 0$ and $r_- < 0$, such that \eqref{hyp - -} holds,

\item if $r_+ > 0$ and $r_- < 0$, such that 
\eqref{hyp + -} holds,

\item if $r_+ < 0$ and $r_- > 0$, such that 
\eqref{hyp - +} holds,
\end{itemize}
then, $F_1 \in \L(X)$, given by \eqref{F1}, is invertible with bounded inverse.
\end{Prop}
\begin{proof}
For a given $\theta \in (0,\pi)$, we have $f_1, \tilde{f_1} \in H(S_{\theta})$. Moreover, for $z \in \CC \setminus (-\infty,r]$, since
$$k_j b_1^{-1}(z) \left(-\sqrt{z+r_+}\right)^{l_j} \left(-\sqrt{z+r_-}\right)^{m_j} \left(-\sqrt{z}\right)^{n_j},\quad \text{for all }j \in J,$$
are polynomial functions, we deduce that $1-\tilde{f_1} \in \mathcal{E}_\infty (S_{\theta})$.

From \eqref{f < 0}, \refP{Prop B1 inv} and \refR{Rem équivalences}, we know that $f_1 < 0$ and $b_2>0$ on $(r,+\infty)$. Then, since $u_{d,r_+},u_{c,r_-},v_{d,r_+},v_{c,r_-} > 0$ on $(r,+\infty)$ and due to \eqref{b1} and \eqref{f(z)}, we deduce that $\tilde{f_1} < 0$ on $(r,+\infty)$. 

Therefore, noting $\tilde{f_{r,1}}(z) = \tilde{f_1}(z+r)$ and  applying \refL{Lem inv} with $G = \tilde{f_1}$ and operator $P=-A-rI$, thus we deduce that operator $F_1 = \tilde{f_{r,1}}(-A-rI) = \tilde{f_1}(-A)$ is invertible with bounded inverse.
\end{proof}
This result finally leads us to state the following main result of this section. 
\begin{Prop}\label{Prop det lambda1 inv}
Assume that $(H_1)$, $(H_2)$, $(H_3)$, $(H_4)$ hold and $k_+k_- >0$. Thus
\begin{itemize}
\item if $r_+ > 0$ and $r_- > 0$,

\item if $r_+ < 0$ and $r_- < 0$, such that \eqref{hyp - -} holds,

\item if $r_+ > 0$ and $r_- < 0$, such that \eqref{hyp + -} holds,

\item if $r_+ < 0$ and $r_- > 0$, such that \eqref{hyp - +} holds,
\end{itemize}
then $\det(\Lambda_1)$ is invertible with bounded inverse.
\end{Prop}
\begin{proof}
From \eqref{det}, \refP{Prop B1 inv} and \refP{prop inv F}, it follows that $\det(\Lambda_1) = -W^{-2}B_1F_1$, is invertible with bounded inverse.
\end{proof}

\subsubsection{Second case}

Let $r_+ \in \RR\setminus\{0\}$ and $r_- = 0$. In the same way than previously, using functional calculus, we prove that the determinant of system \eqref{syst trans M}, given by \eqref{det lambda 2}, is invertible with bounded inverse. Due to \refL{Lem D3+ D3-}, and the definition of $D_4$, we obtain:
$$D_3^+ = g_3^+(-A), \quad D_3^- = g_3^-(-A) \quad \text{and} \quad D_4 = g_4(-A),$$ 
where, for $z \in \CC \setminus \RR_-$, we have set
$$\left\{ \begin{array}{lcl}
g_3^+ (z) & = & \dis 4 k_+^2 (\sqrt{z + r_+} + \sqrt{z})^2 u_{d,r_+}^{-2}(z) v_{d,r_+}^{-2}(z) g_{d,r_+}(z) \\ \ecart

g_3^- (z) & = & \dis 16 k_-^2 \,z\, u_{c,0}^{-2}(z) v_{c,0}^{-2}(z) g_{c,0}(z) \\ \ecart

g_4(z) & = & \dis - 2 \sqrt{z}\left(k_+ f_{d,r_+,3} (z) k_- f_{c,0,2}(z) + k_+ f_{d,r_+,2}(z) k_- f_{c,0,3}(z)\right) \vspace{0.1cm} \\
&&\dis + 2 z k_+ f_{d,r_+,1}(z) k_- f_{c,0,1}(z),
\end{array}\right.$$
with $u_{\delta,r}$, $v_{\delta,r}$, $g_{\delta,r}$ and $f_{\delta,r,i}$ the complex functions defined in section~\ref{sect funct calc}. Thus
\begin{equation} \label{f2(-A)}
\det(\Lambda_2) = D_3^+ + D_3^- + D_4 = f_2(-A),
\end{equation}
with $f_2 = g_3^+ + g_3^- + g_4$. Note that, for some $\theta \in (0,\pi)$, we have $f_2 \in H(S_\theta)$ and due to \refR{Rem f1 +- > 0 et f2,f3 +- < 0} and \refL{Lem g < 0}, for $x> \max(-r_+,0)$, we have 
\begin{equation}\label{f2}
f_2(x)= g_3^+ (x) + g_3^- (x) + g_4 (x), \quad \text{where} \quad g_3^+ < 0 \quad \text{and} \quad g_3^-,g_4 > 0.
\end{equation}

\begin{Lem}\label{Lem f2 > 0}
Let $k_+k_- > 0$. Then
\begin{itemize}
\item if $r_+ > 0$ such that 
\begin{equation}\label{hyp r+ > 0}
r_+ \geqslant \frac{\left(\sqrt{t+1} + \sqrt{t}\right)^2}{t^2} \frac{k_+^2}{4 k_-^2}, \quad \text{for } t> 0~\text{fixed}.
\end{equation}
for all $x \geqslant tr_+$, we have  $f_2(x) > 0$.

\item if $r_+ < 0$ such that 
\begin{equation}\label{hyp r+ < 0}
r_+ \leqslant - \frac{27 k_+^2}{64k_-^2},
\end{equation}
for all $x \geqslant - r_+$, we have $f_2(x) > 0$.
\end{itemize}
\end{Lem}

\begin{proof}
From \eqref{f2}, we deduce
$$f_2(x) \geqslant g_3^+(x) + g_4(x) \geqslant g_3^+ (x) + 2 k_+k_-\, x \,f_{d,r_+,1}(x) f_{c,0,1}(x).$$

Let $r = \max(-r_+,0)$. For $x \in (r,+\infty)$, setting $y=x-r > 0$ and noting 
$$h_1(y) = g_3^+ (y+r) + 2 k_+k_-\, (y+r) \,f_{d,r_+,1}(y+r) f_{c,0,1}(y+r),$$
it follows
$$\begin{array}{lll}
h_1(y) & = & \dis 4 k_+^2 \frac{(\sqrt{y+r + r_+} + \sqrt{y+r})^2}{u_{d,r_+}^{2}(y+r) v_{d,r_+}^{2}(y+r)} g_{d,r_+}(y+r) \\ \ecart
&& \dis + 2 k_+k_-\, (y+r) \,f_{d,r_+,1}(y+r) f_{c,0,1}(y+r).
\end{array}$$
Since we have
$$0 > g_{d,r_+}(y+r) \geqslant - \sqrt{y+r + r_+}\left(1 - e^{-2d(\sqrt{y+r + r_+} + \sqrt{y+r})}\right)^2,$$ 
then
$$\begin{array}{lll}
h_1(y) & \geqslant & \dis 4 \frac{(\sqrt{y+r + r_+} + \sqrt{y+r})\sqrt{y+r + r_+}}{u_{d,r_+}^{2}(y+r) v_{d,r_+}^{2}(y+r)u_{c,0} (y+r) v_{c,0} (y+r)}  h_2(y),
\end{array}$$
where 
$$\begin{array}{lll}
h_2(y) & = & \dis - k_+^2 (\sqrt{y+r + r_+} + \sqrt{y+r}) \left(1 - e^{-2d(\sqrt{y+r + r_+} + \sqrt{y+r})}\right)^2 u_{c,0} (y+r) v_{c,0} (y+r) \\ \ecart
&& \dis + k_+k_-\, (y+r) \left(1 - e^{-2c \sqrt{y+r}}\right)^2 v_{d,r_+} (y+r) \left(1 + e^{-d\sqrt{y+r}}\right)\left(1 + e^{-d\sqrt{y+r+r_+}}\right)\\ \ecart
&& \dis + k_+k_-\, (y+r) \left(1 - e^{-2c \sqrt{y+r}}\right)^2 u_{d,r_+} (y+r)\left(1 - e^{-d\sqrt{y+r}}\right)\left(1 - e^{-d\sqrt{y+r+r_+}} \right) \\ \\

& \geqslant & \dis - k_+^2 (\sqrt{y+r + r_+} + \sqrt{y+r}) \left(1 - e^{-2d(\sqrt{y+r + r_+} + \sqrt{y+r})}\right)^2 \left(1 - e^{-2c \sqrt{y+r}}\right)^2 \\ \ecart
&& \dis + k_+k_-\, (y+r) \left(1 - e^{-2c \sqrt{y+r}}\right)^2 h_3(y),
\end{array}$$
with
$$\begin{array}{lll}
h_3(y) & = & \dis v_{d,r_+} (y+r) \left(1 + e^{-d\sqrt{y+r}}\right)\left(1 + e^{-d\sqrt{y+r+r_+}}\right) \\ \ecart
&& \dis + u_{d,r_+} (y+r)\left(1 - e^{-d\sqrt{y+r}}\right)\left(1 - e^{-d\sqrt{y+r+r_+}} \right) 
\end{array}$$
and
$$\begin{array}{lll}
h_3(y) &\hspace{-0.15cm} = &\hspace{-0.15cm} \dis 2 v_{d,r_+} (y+r)\left(1 + e^{-d(\sqrt{y+r + r_+} + \sqrt{y+r})}\right) + 2 u_{d,r_+} (y+r)\left(e^{-d\sqrt{y+r}} + e^{-d\sqrt{y+r + r_+} }\right) \\ \\

&\hspace{-0.15cm} = &\hspace{-0.15cm} \dis 2 \left(1 - e^{-d(\sqrt{y+r + r_+} + \sqrt{y+r})}\right) \left(1 + e^{-d(\sqrt{y+r + r_+} + \sqrt{y+r})}\right) \\ \ecart
&\hspace{-0.15cm}&\hspace{-0.15cm}\dis + 2 \frac{(\sqrt{y+r + r_+} + \sqrt{y+r})^2}{r_+}  \left(e^{-d\sqrt{y+r}} - e^{-d\sqrt{y+r + r_+} }\right)\left(e^{-d\sqrt{y+r}} + e^{-d\sqrt{y+r + r_+} }\right) \\ \\

&\hspace{-0.15cm} = &\hspace{-0.15cm} \dis 2 \left(1 - e^{-2d(\sqrt{y+r + r_+} + \sqrt{y+r})}\right) \\ \ecart
&\hspace{-0.15cm}&\hspace{-0.15cm} \dis + 2 \frac{(\sqrt{y+r + r_+} + \sqrt{y+r})^2}{r_+}  \left(e^{-2d\sqrt{y+r}} - e^{-2d\sqrt{y+r + r_+} }\right).
\end{array}$$
Moreover, for all $y>0$, since we have
$$\frac{ e^{-2d\sqrt{y+r}} - e^{-2d\sqrt{y+r + r_+}} }{r_+} > 0, \quad \text{for }r_+ \in \RR\setminus\{0\},$$
we deduce that 
$$h_3(y) >  2 \left(1 - e^{-2d(\sqrt{y+r + r_+} + \sqrt{y+r})}\right),$$
and
$$\begin{array}{lll}
h_2(y) & > & \dis - k_+^2 (\sqrt{y+r + r_+} + \sqrt{y+r}) \left(1 - e^{-2d(\sqrt{y+r + r_+} + \sqrt{y+r})}\right)^2 \left(1 - e^{-2c \sqrt{y+r}}\right)^2 \\ \ecart
&& \dis + 2k_+k_-\, (y+r) \left(1 - e^{-2c \sqrt{y+r}}\right)^2 \left(1 - e^{-2d(\sqrt{y+r + r_+} + \sqrt{y+r})}\right) \\ \\

& > & \dis - k_+^2 (\sqrt{y+r + r_+} + \sqrt{y+r}) \left(1 - e^{-2d(\sqrt{y+r + r_+} + \sqrt{y+r})}\right) \left(1 - e^{-2c \sqrt{y+r}}\right)^2 \\ \ecart
&& \dis + 2k_+k_-\, (y+r) \left(1 - e^{-2d(\sqrt{y+r + r_+} + \sqrt{y+r})}\right) \left(1 - e^{-2c \sqrt{y+r}}\right)^2  \\ \\

& > & \dis \left(1 - e^{-2d(\sqrt{y+r + r_+} + \sqrt{y+r})}\right) \left(1 - e^{-2c \sqrt{y+r}}\right)^2 h_4(y),
\end{array}$$
where
$$h_4(y) = 2k_+k_-\, (y+r) - k_+^2 (\sqrt{y+r + r_+} + \sqrt{y+r}).$$
Thus
\begin{equation}\label{h4 > 0}
h_4(y) \geqslant 0 \Longleftrightarrow \frac{y+r}{\sqrt{y+r + r_+} + \sqrt{y+r}} \geqslant \frac{k_+^2}{2k_+k_-}.
\end{equation}
We set
\begin{equation}\label{h5}
h_5(y) = \frac{y+r}{\sqrt{y+r + r_+} + \sqrt{y+r}},
\end{equation}
hence
$$h'_5(y) = \left(\frac{1}{\sqrt{y+r + r_+} + \sqrt{y+r}} \right)\left(1 - \frac{1}{2}\sqrt{\frac{y+r}{y+r+r_+}}\right).$$
\begin{enumerate}
\item If $r_+<0$, then $r=-r_+$ and
$$\frac{y+r}{y+r+r_+} = \frac{y + r}{y},$$
moreover
$$h'_5(y) \geqslant 0 \Longleftrightarrow 1 - \frac{1}{2}\sqrt{\frac{y + r}{y}} \geqslant 0 \Longleftrightarrow 4 \geqslant \frac{y + r}{y} \Longleftrightarrow y \geqslant \frac{r}{3}.$$
Thus, we have
$$h_5(y) \geqslant h_5\left(\frac{r}{3}\right) = \frac{\frac{4r}{3}}{\sqrt{\frac{r}{3}} + 2\sqrt{\frac{r}{3}}} = \frac{4}{3}\sqrt{\frac{r}{3}} > 0.$$
This yields that for all $y\geqslant 0$, we have
$$h_5(y) \geqslant\frac{4}{3}\sqrt{\frac{r}{3}} > 0.$$
Therefore, from \eqref{h4 > 0} and \eqref{h5}, we deduce that
$$h_4(y) \geqslant 0 \Longleftrightarrow h_5(y) \geqslant \frac{k_+^2}{2k_+k_-} \Longleftrightarrow \frac{4}{3}\sqrt{\frac{r}{3}} \geqslant \frac{k_+}{2k_-},$$
hence, since $r=-r_+ > 0$, we obtain 
$$\frac{4}{3}\sqrt{\frac{r}{3}} \geqslant \frac{k_+}{2k_-} \Longleftrightarrow \sqrt{\frac{r}{3}} \geqslant \frac{3k_+}{8k_-} \Longleftrightarrow \frac{r}{3} \geqslant \frac{9k_+^2}{64k_-^2} \Longleftrightarrow -r_+ \geqslant \frac{27 k_+^2}{64k_-^2}.$$

\item If $r_+>0$, then $r = 0$ and 
$$\frac{y+r}{y+r+r_+} = \frac{y}{y+r_+} < 1,$$
hence $h'_5 > 0$ and $h_5$ is an increasing function. Thus, from \eqref{h5}, since $r=0$, it follows that
$$h_5(y) = \frac{y}{\sqrt{y + r_+} + \sqrt{y}},$$
then
$$h_5(y) \geqslant \frac{k_+^2}{2k_+k_-} \Longleftrightarrow \frac{y}{\sqrt{y + r_+} + \sqrt{y}} \geqslant \frac{k_+}{2k_-}.$$
Moreover, for $t> 0$ fixed, we have
$$h_5(tr_+) \geqslant \frac{k_+}{2k_-} \Longleftrightarrow \frac{tr_+}{\sqrt{(t+1) r_+} + \sqrt{tr_+}} \geqslant \frac{k_+}{2k_-} \Longleftrightarrow \frac{t}{\sqrt{t+1} + \sqrt{t}} \sqrt{r_+} \geqslant \frac{k_+}{2k_-},$$
hence
$$\sqrt{r_+} \geqslant \frac{\sqrt{t+1} + \sqrt{t}}{t} \frac{k_+}{2k_-} \Longleftrightarrow r_+ \geqslant \frac{\left(\sqrt{t+1} + \sqrt{t}\right)^2}{t^2} \frac{k_+^2}{4 k_-^2}.$$
\end{enumerate}
Finally, if $r_+ > 0$ such that \eqref{hyp r+ > 0} holds, then since $y=x$, for all $x \geqslant tr_+$, we have $h_2 (x) > 0$, $h_1(x) > 0$ and $f_2(x)> 0$. Moreover, $r_+ < 0$ such that \eqref{hyp r+ < 0} holds, then for all $y > 0$, we have $h_2 (y) > 0$, $h_1(y) > 0$ and since $y=x + r_+$, for all $x > -r_+$, it follows that $f_2(x) > 0$.
\end{proof}
Therefore, as in the first case, since we have $U_\pm \sim I$ and $V_\pm \sim I$, then by setting $W=U_- U_+ V_- V_+ \sim I$, we deduce that
$$\left\{\begin{array}{lcllcl}
W P_1^+ & \sim & 2k_+ (L_+ + M), & W Q_1^- & \sim & 2k_- I \\ \ecart
 
W P_2^+ & \sim & 2k_+ (L_+ + M), & W Q_2^- & \sim & 2k_- I \\ \ecart
 
W P_3^+ & \sim & 2k_+ (L_+ + M) L_+, & W Q_3^- & \sim & 2k_- I.
\end{array}\right.$$
Thus
$$\begin{array}{rcl}
W^2 \det(\Lambda_2) & = & \dis \left(W P_2^+ W P_3^+ - M \left(W P_1^+\right)^2 \right) + 4 M^2 \left(W Q_2^- W Q_3^- - M \left(W Q_1^-\right)^2 \right) \\ \ecart
&& \dis + 2M \left(W P_3^+ W Q_2^- + W P_2^+ W Q_3^- + M W P_1^+ W Q_1^- \right) \\ \\

& \sim & \dis 4 k_+^2 (L_+ + M)^2 (L_+ - M) + 16 k_-^2 M^2 (I - M) \\ \ecart
&& \dis + 8 k_+ k_- (L_+ + M) M (L_+ + M + I).
\end{array}$$
From \eqref{L-M=}, we have
$$\begin{array}{rcl}
W^2 \det(\Lambda_2) & \sim & \dis 4 k_+^2 r_+ (L_+ + M) + 16 k_-^2 M^2 (I - M) \\ \ecart
&& \dis + 8 k_+ k_- (L_+ + M) M (L_+ + M + I) \\ \\

& \sim & \dis 4 k_+ l_+ (L_+ + M) + 16 k_-^2 M^2 (I - M) \\ \ecart
&& \dis + 8 k_+ k_- (L_+ + M) M (L_+ + M + I)
\end{array}$$
Hence, we note
$$B_3 = 4 k_+ l_+ (L_+ + M) + 16 k_-^2 M^2 (I - M) + 8 k_+ k_- (L_+ + M) M (L_+ + M + I).$$
Thus, we obtain 
\begin{equation}\label{det lambda2 =  W-2 BF}
\det(\Lambda_2) = W^{-2} \left(B_3 + \sum_{j \in J} k_j L_+^{l_j} M^{m_j} e^{\alpha_j L_+} e^{\beta_j M}\right),
\end{equation}
where $J$ is a finite set and for any $j \in J$:
$$k_j \in \RR;~l_j, m_j \in \NN,~ \alpha_j, \beta_j \in \RR_+ \quad \text{with} \quad \alpha_j + \beta_j \neq 0.$$
\begin{Prop}\label{Prop B5 inv}
Assume that $(H_1)$, $(H_2)$, $(H_3)$ hold and $k_+k_- > 0$. If $\dis\frac{k_-}{k_+} \leqslant 2$, then 
$$0 \in \rho\left(8 k_+k_- (L_+ + M)^2 M - 16 k_-^2 M^3\right).$$
\end{Prop}
\begin{proof}
Since $k_+k_- > 0$, we have  $\dis \frac{k_-}{k_+} > 0$ and 
$$8 k_+k_- (L_+ + M)^2 M - 16 k_-^2 M^3 = 8 k_+k_- M \left[L_+^2 + 2L_+ M + M^2 - 2 \frac{k_-}{k_+} M^2\right].$$
From \refR{Remconsequences}, statement 5 and \cite{pruss-sohr}, Corollary 3, p. 444, we deduce that
$$L_+^2,~ 2L_+M,~ M^2 \in \text{BIP}(X,\theta_A).$$  
Thus, if $\dis\frac{k_-}{k_+} \leqslant 1$, then
$$L_+^2 + 2L_+ M + M^2 - 2 \frac{k_-}{k_+} M^2 = L_+^2 - \frac{k_-}{k_+} M^2 + 2L_+ M + \left(1- \frac{k_-}{k_+} \right) M^2.$$
Moreover, for all $\psi \in D(M^2)=D(A)$, due to \eqref{L2-M2}, we have
$$\left(L_+^2 - \frac{k_-}{k_+} M^2\right)\psi = \left[-\left(1-\frac{k_-}{k_+}\right) A + r_+ I\right]\psi,$$
and from \cite{pruss-sohr}, Theorem 3, p. 437 and \cite{arendt-bu-haase}, Theorem 2.3, p. 69, assumptions $(H_2)$ and $(H_3)$ imply that 
$$-\left(1-\frac{k_-}{k_+}\right) A + r_+ I \in \text{BIP}(X,\theta_A),$$
and
$$L_+^2 - \frac{k_-}{k_+} M^2 + 2L_+ M + \left(1- \frac{k_-}{k_+} \right) M^2 \in \text{BIP}(X,\theta_A + \varepsilon),$$
for any $\varepsilon \in (0, \pi-\theta_A)$. Moreover, since $0 \in \rho(L_+M)$, we deduce from \cite{pruss-sohr}, remark at the end of p. 445, that $0\in \rho\left(L_+^2 - \frac{k_-}{k_+} M^2 + 2L_+ M + \left(1- \frac{k_-}{k_+} \right) M^2\right)$. Therefore, since $0 \in \rho(M)$ and $k_+k_- > 0$, it follows that
$$0 \in \rho\left(8 k_+k_- (L_+ + M)^2 M - 16 k_-^2 M^3\right).$$
In the same way, if $1 < \dis\frac{k_-}{k_+} \leqslant 2$, then
$$L_+^2 + 2L_+ M + M^2 - 2 \frac{k_-}{k_+} M^2 = L_+^2 - M^2 + 2L_+ M - 2\left(\frac{k_-}{k_+} - 1\right) M^2 + M^2 - M^2,$$
hence, for all $\psi \in D(M^2) = D(A)$, from \eqref{L2-M2}, we obtain
\begin{equation}\label{k-/k+ < 2}
\left(L_+^2 + 2L_+ M + M^2 - 2 \frac{k_-}{k_+} M^2\right)\psi = r_+ \psi + 2M \left(L_+ - \left(\frac{k_-}{k_+} - 1\right) M\right) \psi.
\end{equation}
Moreover, we have
$$\left(L_+ - \left(\frac{k_-}{k_+} - 1\right)M\right)\psi = \left(L_+ + \left(\frac{k_-}{k_+} - 1\right)M\right)^{-1} \left(L_+^2 - \left(\frac{k_-}{k_+} - 1\right)^2 M^2\right)\psi,$$
and from \cite{pruss-sohr}, Theorem 3, p. 437 and \cite{arendt-bu-haase}, Theorem 2.3, p. 69, assumptions $(H_2)$ and $(H_3)$ imply that 
\begin{equation}\label{k-/k+ < 2 BIP}
\left(L_+^2 - \left(\frac{k_-}{k_+} - 1\right)^2 M^2\right) = -\left(2 - \frac{k_-}{k_+}\right) A + r_+ I \in \text{BIP}(X,\theta_A).
\end{equation}
Finally, from $(H_2)$, $(H_3)$, \eqref{k-/k+ < 2}, \eqref{k-/k+ < 2 BIP} and \cite{pruss-sohr}, Theorem 3, p. 437, we deduce that
$$r_+ \psi + 2M \left(L_+ - \left(\frac{k_-}{k_+} - 1\right) M\right) \in \text{BIP}(X,\theta_A),$$
and 
$$0 \in \rho\left(r_+ \psi + 2M \left(L_+ - \left(\frac{k_-}{k_+} - 1\right) M\right)\right).$$
Therefore, since $0 \in \rho(M)$ and $k_+k_- > 0$, it follows that
$$0 \in \rho\left(8 k_+k_- (L_+ + M)^2 M - 16 k_-^2 M^3\right).$$
\end{proof}
We set 
\begin{equation}\label{B4}
\begin{array}{lll}
B_4 & = & \dis I + 4 k_+ l_+ (L_+ + M) \left(8 k_+k_- (L_+ + M)^2 M - 16 k_-^2 M^3\right)^{-1} \\ \ecart

&& \dis + 16 k_-^2 M^2 \left(8 k_+k_- (L_+ + M)^2 M - 16 k_-^2 M^3\right)^{-1} \\ \ecart
&& \dis + 8 k_+ k_- (L_+ + M) M  \left(8 k_+k_- (L_+ + M)^2 M - 16 k_-^2 M^3\right)^{-1},
\end{array}
\end{equation}
thus, we have
$$B_3 = \left(8 k_+k_- (L_+ + M)^2 M - 16 k_-^2 M^3\right) B_4.$$
Moreover, from \eqref{det lambda2 =  W-2 BF} and noting $B_5 = 8 k_+k_- (L_+ + M)^2 M - 16 k_-^2 M^3$, we have 
\begin{equation}\label{det Lammbda2 = W-2 B5 F2}
\det(\Lambda_2) = W^{-2} B_5 F_2,
\end{equation}
where
\begin{equation}\label{F2}
F_2 = B_4 + \sum_{j \in J} k_j B_5^{-1} L_+^{l_j} M^{m_j} e^{\alpha_j L_+} e^{\beta_j M}.
\end{equation}
Now, for $z \in \CC \setminus [\max(-r_+,0),+\infty)$, we set
$$\tilde{f_2}(z) = b_4(z) + \sum_{j\in J} k_j b_5(z)^{-1} \sqrt{z+r_+}^{l_j} \sqrt{z}^{m_j} e^{-\alpha_j \sqrt{z+r_+}} e^{-\beta_j \sqrt{z}},$$
where $b_3(z) = b_4(z)b_5(z)$, with
$$\begin{array}{lll}
b_4(z) & = & 1 + 4 k_+ l_+ (\sqrt{z+r_+} + \sqrt{z}) \left(8 k_+k_- (\sqrt{z + r_+} + \sqrt{z})^2 \sqrt{z} - 16 k_-^2 \sqrt{z}^3\right)^{-1} \\ \ecart

&& \dis + 16 k_-^2 z \left(-8 k_+k_- (\sqrt{z + r_+} + \sqrt{z})^2 \sqrt{z} + 16 k_-^2 \sqrt{z}^3\right)^{-1} \\ \ecart
&& \dis + 8 k_+ k_- (\sqrt{z+r_+} + \sqrt{z}) \sqrt{z} \left(8 k_+k_- (\sqrt{z + r_+} + \sqrt{z})^2 \sqrt{z} - 16 k_-^2 \sqrt{z}^3\right)^{-1},
\end{array}$$
and
$$b_5(z) = -8 k_+k_- (\sqrt{z + r_+} + \sqrt{z})^2 \sqrt{z} + 16 k_-^2 \sqrt{z}^3.$$
Then $\tilde{f_2}(-A) = F_2$, $b_3(-A) = B_3$, $b_4(-A) = B_4$ and $b_5(-A) = B_5$. Thus, from \eqref{f2(-A)} and \eqref{det Lammbda2 = W-2 B5 F2}, we deduce that
\begin{equation}\label{f2(z)}
f_2(z) = u_{d,r_+}^{-2}(z)v_{d,r_+}^{-2}(z)u_{c,0}^{-2}(z)v_{c,0}^{-2}(z)b_5(z) \tilde{f_2}(z).
\end{equation}
\begin{Prop}\label{Prop F2 inv}
Assume that $(H_1)$, $(H_2)$, $(H_3)$, $(H_4)$ hold and $k_+k_- > 0$ with $\dis\frac{k_-}{k_+} \leqslant 2$. Thus
\begin{itemize}
\item if $r_+ > 0$ such that 
\begin{equation}\label{hyp r+ > 0 F2 inv}
r_+ \geqslant \frac{\left(\sqrt{t+1} + \sqrt{t}\right)^2}{t^2} \frac{k_+^2}{4 k_-^2}, \quad \text{for } t\in \left(0,\frac{1}{r_+\|A^{-1}\|_{\L(X)}} \right)~\text{fixed},
\end{equation}

\item if $r_+ < 0$ such that \eqref{hyp r+ < 0} holds, 
\end{itemize}
then $F_2$, given by \eqref{F2}, is invertible with bounded inverse.
\end{Prop}

\begin{proof}
From \refP{Prop B5 inv} and \eqref{F2}, we deduce that $F_2$ is well defined. 
\begin{itemize}
\item Assume that $r_+ > 0$ such that \eqref{hyp r+ > 0} holds. Then, from \refL{Lem f2 > 0} and \eqref{f2(z)}, it follows that $f_2$ does not vanish on $(tr_+,+\infty)$, for $t> 0$ fixed, which involves that $u_{d,r_+}^{-2}$, $v_{d,r_+}^{-2}$, $u_{c,0}^{-2}$, $v_{c,0}^{-2}$, $b_5$ and $\tilde{f_2}$ do not vanish on $(tr_+,+\infty)$, for $t> 0$ fixed. Moreover, due to $(H_2)$, there exists $R= \frac{1}{\|A^{-1}\|_{\L(X)}} > 0$ such that $B(0,R) \subset \rho(A)$. Therefore, setting $\tilde{f}_{tr_+,2}(z) = \tilde{f_2}(z+tr_+)$, with $t \in \left(0, \frac{1}{r_+\|A^{-1}\|_{\L(X)}}\right)$ fixed and applying \refL{Lem inv} where we have set $G = \tilde{f}_{t r_+,2}$ and operator $P=-A-tr_+I \in$ BIP\,$(X,\theta_A)$ (due to $(H_2)$ and $(H_3)$), we deduce that operator $F_2 = \tilde{f}_{tr_+,2}(-A-tr_+I) = \tilde{f_2}(-A)$ is invertible with bounded inverse.

\item Now, assume that $r_+ < 0$ such that \eqref{hyp r+ < 0} holds. Then $\tilde{f_2}$ does not vanish on $(-r_+,+\infty)$. Moreover, from $(H_2)$ and $(H_3)$, we have $-A + r_+ I \in$ BIP\,$(X,\theta_A)$. It follows that $F_2 = \tilde{f}_{-r_+,2}(-A+r_+I) = \tilde{f_2}(-A)$ is invertible with bounded inverse. 
\end{itemize}
\end{proof}
This result finally leads us to state the following main result of this section.
\begin{Prop}\label{Prop det lambda2 inv}
Assume that $(H_1)$, $(H_2)$, $(H_3)$, $(H_4)$ hold and $k_+k_- >0$ with $\dis\frac{k_-}{k_+} \leqslant 2$. Thus
\begin{itemize}
\item if $r_+ > 0$ such that \eqref{hyp r+ > 0 F2 inv} holds, 

\item if $r_+ < 0$ such that \eqref{hyp r+ < 0} holds, 
\end{itemize}
then $\det(\Lambda_2)$ is invertible with bounded inverse.
\end{Prop}
\begin{proof}
From \eqref{det Lammbda2 = W-2 B5 F2}, \refP{Prop B5 inv} and \refP{Prop F2 inv}, we obtain that $\det(\Lambda_2) = W^{-2}B_5F_2$, is invertible with bounded inverse.
\end{proof}

\subsection{Regularity}

\subsubsection{First case}

Here, we consider $r_+,r_- \in \RR \setminus\{0\}$. From \refT{Th syst trans}, we have to prove that system \eqref{syst trans gén} has a unique solution $(\psi_1,\psi_2)$ satisfying \eqref{reg (psi1,psi2)}. The existence and uniqueness of this solution is ensured by \refP{Prop det lambda1 inv}, so we have
\begin{equation}\label{syst psi1 psi2 1}
\left\{\begin{array}{lll}
\psi_1 & = & \dis \left(P_1^- - P_1^+\right)\left[\det(\Lambda_1)\right]^{-1} S_1 - \left(P_2^+ + P_2^- \right)\left[\det(\Lambda_1)\right]^{-1} S_2 \\ \\

\psi_2 & = & \dis - \left(P_3^+ - P_3^- \right)\left[\det(\Lambda_1)\right]^{-1} S_1 + M \left(P_1^- - P_1^+ \right) \left[\det(\Lambda_1)\right]^{-1} S_2.
\end{array}\right.
\end{equation}
Now, we have to study the regularity of $\left[\det(\Lambda_1)\right]^{-1}$. Since, in this case, the determinant $\det(\Lambda_1)$ is the same than the one in \cite{LLMT}, we deduce, from \cite{LLMT}, Lemma 5.3, p. 2958, that there exists $R_{\det(\Lambda_1)} \in \L(X)$ such that
$$R_{\det(\Lambda_1)}(X) \subset D(M), \quad \left[\det(\Lambda_1)\right]^{-1} = N^{-1} + N^{-1}R_{\det(\Lambda_1)},$$
where $N = 4 k_+k_- (L_- + M)(L_+ + M)(L_+ + L_- + 2M)$.
Then, the rest of the proof is similar to the one given in \cite{LLMT}, section 5.3. Therefore, from \eqref{S2} and \eqref{S1}, it follows that $S_1, S_2 \in (D(M),X)_{1 + \frac{1}{p},p}$ and thus
\begin{equation}\label{det-1 S1 S2}
\left[\det(\Lambda)\right]^{-1}S_1,\left[\det(\Lambda)\right]^{-1}S_2 \in (D(M),X)_{4+\frac{1}{p},p}.
\end{equation}
Moreover, from \eqref{syst psi1 psi2 1}, we have
\begin{equation}\label{syst psi1 psi2 2}
\left\{\begin{array}{lll}
\psi_1 & = & \dis -2\left(k_+ (L_+ + M) - k_- (L_- + M)\right) \left[\det(\Lambda_1)\right]^{-1} S_1 \\ \ecart
&& \dis + 2\left( k_+ (L_+ + M) - k_- (L_- + M) \right) \left[\det(\Lambda_1)\right]^{-1} S_2 + \tilde{S}_1 \\ \\

\psi_2 & = & \dis -2\left(k_+ (L_+ + M)L_+ + k_- (L_- + M)L_-\right) \left[\det(\Lambda_1)\right]^{-1} S_1 \\ \ecart
&& \dis - 2\left(k_+ (L_+ + M) - k_- (L_- + M)\right) \left[\det(\Lambda_1)\right]^{-1} S_2 + \tilde{S}_2,
\end{array}\right.
\end{equation}
where $\tilde{S}_1$, $\tilde{S}_2 \in D(M^\infty)$. Finally, from \eqref{égalités espaces interpolations}, \eqref{det-1 S1 S2} and \eqref{syst psi1 psi2 2}, we obtain 
$$\left\{\begin{array}{l}
\psi_1 \in (D(M),X)_{3 + \frac{1}{p},p} = (D(A),X)_{1 + \frac{1}{2p},p} \\ \ecart
\psi_2 \in (D(M),X)_{2 + \frac{1}{p},p} = (D(A),X)_{1 + \frac{1}{2} + \frac{1}{2p},p}.
\end{array}\right.$$

\subsubsection{Second case}

Here, we consider $r_+ \in \RR \setminus\{0\}$ and $r_- = 0$. From \refT{Th syst trans M}, we have to prove that \eqref{syst trans M} has a unique solution $(\psi_1,\psi_2)$ satisfying \eqref{reg (psi1,psi2)}. The existence and uniqueness of this solution is ensured by \refP{Prop det lambda2 inv}, so we have
\begin{equation}\label{syst psi1 psi2 1 M}
\left\{\begin{array}{lll}
\psi_1 & = & \dis \left(2M Q_1^- - P_1^+ \right)\left[\det(\Lambda_2)\right]^{-1} S_3 + \left(P_2^+ + 2M Q_2^-\right) \left[\det(\Lambda_2)\right]^{-1} S_4 \\ \\

\psi_2 & = & \dis -\left(P_3^+ + 2M Q_3^- \right)\left[\det(\Lambda_2)\right]^{-1} S_3 + M \left(P_1^+ - 2M Q_1^- \right) \left[\det(\Lambda_2)\right]^{-1} S_4.
\end{array}\right.
\end{equation}
Now, we have to study the regularity of $\left[\det(\Lambda_2)\right]^{-1}$. From \eqref{U V +-}, \eqref{B4}, \eqref{det Lammbda2 = W-2 B5 F2}, \eqref{F2} and \cite{LMMT}, Lemma 5.1, p. 365, we deduce that there exists $R_{\det(\Lambda_2)} \in \L(X)$ such that
$$R_{\det(\Lambda_2)}(X) \subset D(M), \quad \left[\det(\Lambda_2)\right]^{-1} = B_5^{-1} + B_5^{-1}R_{\det(\Lambda_2)},$$
where we recall that $B_5 = 8 k_+k_- (L_+ + M)^2 M - 16 k_-^2 M^3$. Moreover, from \eqref{reg phi +-}, \eqref{égalités espaces interpolations}, \eqref{phi i tilde +} and \eqref{phi i tilde - M}, we have
$$\tilde{\varphi}_1^+, \tilde{\varphi}_2^-,\tilde{\varphi}_2^+,\tilde{\varphi}_3^+, \tilde{\varphi}_4^-, \tilde{\varphi}_4^+ \in \left(D(M),X\right)_{2+\frac{1}{p},p} \quad \text{and} \quad \tilde{\varphi}_1^-,\tilde{\varphi}_3^- \in \left(D(M),X\right)_{3+\frac{1}{p},p}.$$
Thus, from \eqref{S3}, \eqref{S4}, \eqref{R2}, \refR{Rem F+-} and \refR{Rem tilde F-}, we deduce that
$$R_2 \in \left(D(M),X\right)_{\frac{1}{p},p} \quad \text{and} \quad S_3, S_4 \in \left(D(M),X\right)_{1+\frac{1}{p},p},$$
which implies that
\begin{equation}\label{det-1 S3 S4}
\left[\det(\Lambda_2)\right]^{-1} S_3,~\left[\det(\Lambda_2)\right]^{-1} S_4 \in  \left(D(M),X\right)_{4+\frac{1}{p},p}.
\end{equation}
Moreover, due to \eqref{Pi+}, \eqref{Qi-}, \eqref{syst psi1 psi2 1 M} and \cite{LMMT}, Lemma 5.1, p. 365, we have
\begin{equation}\label{syst psi1 psi2 2 M}
\left\{\begin{array}{lll}
\psi_1 & = & \dis -2\left(k_+ (L_+ + M) - 2k_- M\right) \left[\det(\Lambda_2)\right]^{-1} S_3 \\ \ecart
&& \dis + 2\left( k_+ (L_+ + M) + 2 k_- M \right) \left[\det(\Lambda_2)\right]^{-1} S_4 + \tilde{S}_3 \\ \\

\psi_2 & = & \dis -2\left(k_+ (L_+ + M)L_+ + 2 k_- M\right) \left[\det(\Lambda_2)\right]^{-1} S_3 \\ \ecart
&& \dis + 2 M \left(k_+ (L_+ + M) - 2 k_- M\right) \left[\det(\Lambda_2)\right]^{-1} S_4 + \tilde{S}_4,
\end{array}\right.
\end{equation}
where $\tilde{S}_3$, $\tilde{S}_4 \in D(M^\infty)$. Finally, from \eqref{égalités espaces interpolations}, \eqref{det-1 S3 S4} and \eqref{syst psi1 psi2 2 M}, we obtain 
$$\left\{\begin{array}{l}
\psi_1 \in (D(M),X)_{3 + \frac{1}{p},p} = (D(A),X)_{1 + \frac{1}{2p},p} \\ \ecart
\psi_2 \in (D(M),X)_{2 + \frac{1}{p},p} = (D(A),X)_{1 + \frac{1}{2} + \frac{1}{2p},p}.
\end{array}\right.$$

\section*{Acknowledgments} 

I would like to thank the referee for the valuable comments and corrections which have improved this paper.


\begin{thebibliography}{99}

\bibitem{arendt-bu-haase} {\sc W.~Arendt, S.~Bu \& M.~Haase}, \textit{Functional calculus, variational methods and Liapunov's theorem}, Arch. Math., 77 (2001), 65-75.

\bibitem{ACT} {\sc  M.A. Aziz-Alaoui, G. Cantin \& A. Thorel}, \textit{Synchronization of Turing patterns in complex networks of reaction–diffusion systems set in distinct domains}, Nonlinearity, 37 (2024), 025011 (32pp).

\bibitem{BBT} {\sc M. Benharrat, F. Bouchelaghem \& A. Thorel}, \textit{On the solvability of fourth-order boundary value problems with accretive operators}, Semigroup Forum, 107 (2023), 17-39.

\bibitem{bourgain} {\sc J.~Bourgain}, \textit{Some remarks on Banach spaces in which martingale difference sequences are unconditional}, Ark. Mat., 21 (1983), 163-168.

\bibitem{burkholder} {\sc D.L.~Burkholder}, \textit{A geometrical characterisation of Banach spaces in which martingale difference sequences are unconditional}, Ann. Probab., 9 (1981), 997-1011. 

\bibitem{cantin-thorel} {\sc G. Cantin \& A. Thorel}, \textit{On a generalized diffusion problem: a complex network approach}, Discrete Contin. Dyn. Syst. - B, 27, 4 (2022), 2345-2365. 

\bibitem{Cantrell-Cosner} {\sc R.S. Cantrell \& C. Cosner}, ``Spatial Ecology via Reaction-Diffusion Equations'', Wiley, 2003.

\bibitem{cohen-murray} {\sc D.S.~Cohen \& J.D.~Murray}, \textit{A generalized diffusion model for growth and dispersal in population}, Journal of Mathematical Biology, 12 (1981), 237-249.

\bibitem{da prato-grisvard} {\sc G.~Da Prato \& P.~Grisvard}, \textit{Somme d'op\'erateurs lin\'eaires et \'equations diff\'erentielles op\'erationnelles}, J. Math. Pures et Appl., 54 (1975), 305-387.   

\bibitem{DMT} {A. Ducrot, P. Magal \& A. Thorel}, \textit{An integrated semigroup approach for age structured equations with diffusion and non-homogeneous boundary conditions}, Nonlinear Differ. Equ. Appl., 28, 49 (2021).

\bibitem{denk} {\sc R. Denk \& F. Kammerlander}, \textit{Exponential stability for a coupled system of damped-undamped plate equations}, IMA J. Appl. Math., 83 (2018), 302-322.

\bibitem{dore-venni} {\sc G. Dore \& A. Venni}, \textit{On the closedness of the sum of two closed operators}, Math. Z., 196 (1987), 189-201.

\bibitem{hassan}{\sc A. Favini, R. Labbas, S. Maingot, K. Lemrabet \& H. Sidib\'e}, \textit{Resolution and optimal regularity for a biharmonic equation with impedance boundary conditions and some generalizations}, Discrete Contin. Dyn. Syst. - A, 33, 11-12 (2013), 4991-5014.

\bibitem{gilbarg-trudinger} {\sc D.~Gilbarg \& N.~Trudinger}, ``Elliptic partial differential equations of second order'', Classics in Mathematics, Springer-Verlag, Berlin, 2001. 

\bibitem{grisvard} {\sc P. Grisvard}, \textit{Spazi di tracce ed Applicazioni}, Rendiconti di Matematica, Serie VI, 5 (1972), 657-729.

\bibitem{gros} {\sc P. Gros}, ``La repr\'esentation de l'espace dans les mod\`eles de dynamiques de populations, Mod\`eles dynamiques d\'eterministes \`a temps et espace continus'', Ifremer, 2001. 

\bibitem{haase} {\sc M. Haase}, ``The Functional Calculus for Sectorial Operators, Operator Theory: Advances and Applications'', Birkh\"{a}user Verlag, Basel-Boston-Berlin, 2006.

\bibitem{komatsu} {\sc H. Komatsu}, \textit{Fractional Powers of Operators}, Pac. J. Math., 19, 2 (1966), 285-346.

\bibitem{LLMT} {\sc R.~Labbas, K.~Lemrabet, S.~Maingot \& A.~Thorel}, \textit{Generalized linear models for population dynamics in two juxtaposed habitats}, Discrete Contin. Dyn. Syst. - A, 39, 5 (2019), 2933-2960.  

\bibitem{LMMT} {\sc R. Labbas, S. Maingot, D. Manceau \& A. Thorel}, \textit{On the regularity of a generalized diffusion problem arising in population dynamics set in a cylindrical domain}, J. Math. Anal. Appl., 450 (2017), 351-376.

\bibitem{LMT} {\sc R. Labbas, S. Maingot \& A. Thorel}, \textit{Generation of analytic semigroup for some generalized diffusion operators in $L^p$-spaces}, Math. Ann., 384 (2022), 1-49. 

\bibitem{Cone 1} {\sc R. Labbas, S. Maingot \& A. Thorel}, \textit{Solvability of a fourth order elliptic problem in a bounded sector, part I}, Boll. Unione Mat. Ital., 17 (2024), 647-666.

\bibitem{Cone 2} {\sc R. Labbas, S. Maingot \& A. Thorel}, \textit{Solvability of a fourth order elliptic problem in a bounded sector, part II}, Boll. Unione Mat. Ital., (2024).

\bibitem{lions-peetre}  {\sc J.-L.~Lions \& J.~Peetre}, \textit{Sur une classe d'espaces d'interpolation}, Publications math\'ematiques de l'I.H.\'E.S., 19 (1964), 5-68.

\bibitem{lunardi} {\sc A.~Lunardi}, ``Analytic semigroups and optimal regularity in parabolic problems'', Bir-khauser, Basel, Boston, Berlin, 1995.  

\bibitem{lunardi 2} {\sc A.~Lunardi}, ``Interpolation theory'', Third edition, Lecture Notes, Scuola Normale Superiore di Pisa (New Series), 16, Edizioni della Normale, Pisa, 2018. 

\bibitem{murray} {\sc J.D.~Murray}, ``Mathematical Biology II: Spatial Models and Biomedical Applications'', Third Edition, Springer, 2003.  

\bibitem{novick-cohen} {\sc A. Novick-Cohen}, \textit{On Cahn-Hilliard type equations}, Nonlinear Analysis, Theory, Methods \& Applications, 15, 9 (1990), 797-814.

\bibitem{ochoa} {\sc F.L. Ochoa}, \textit{A generalized reaction-diffusion model for spatial structures formed by motile cells}, BioSystems, 17 (1984), 35-50.

\bibitem{okubo} {\sc A. Okubo \& S.A. Levin}, ``Diffusion and ecological problems, Mahematical biology'', second edition, Springer-Verlag, Berlin, 2010.

\bibitem{pruss-sohr} {\sc J.~Pr\"uss \& H.~Sohr}, \textit{On operators with bounded imaginary powers in Banach spaces}, Mathematische Zeitschrift, 203 (1990), 429-452.

\bibitem{pruss-sohr2} {\sc J.~Pr\"uss \& H.~Sohr}, \textit{Imaginary powers of elliptic second order differential operators in $L^p$-spaces}, Hiroshima Math. J., 23, 1 (1993), 161-192. 

\bibitem{rubio} {\sc J.L.~Rubio de Francia}, \textit{Martingale and integral transforms of Banach space valued functions, Probability and Banach Spaces}, (Zaragoza, 1985), in: Lecture Notes in Math., 1221, Springer-Verlag, Berlin (1986), 195-222.

\bibitem{thorel} {\sc A.~Thorel}, \textit{Operational approach for biharmonic equations in $L^p$-spaces}, J. Evol. Equ., 20 (2020), 631-657. 

\bibitem{thorel 2} {\sc A.~Thorel}, \textit{A biharmonic transmission problem in $L^p$-spaces}, Commun. Pure Appl. Anal., 20, 9 (2021), 3193-3213.  

\bibitem{triebel} {\sc H. Triebel}, ``Interpolation Theory, Functions Spaces, Differential Operators'', Amsterdam, N.Y.,Oxford, North-Holland, 1978.

\end{thebibliography}
\end{document}